\documentclass[final,reqno]{amsart}

\usepackage{fullpage}
\usepackage{mathtools,amssymb,amsthm}
\usepackage[foot]{amsaddr}
\usepackage{mathrsfs,euscript}
\usepackage{bm} 
\usepackage[only llbracket,rrbracket,longmapsfrom]{stmaryrd}
\usepackage{enumitem}
\usepackage{showkeys}
\usepackage[l2tabu,orthodox]{nag}
\usepackage[bookmarks=false,draft=false,breaklinks,colorlinks]{hyperref} 
\usepackage{cleveref}

\numberwithin{equation}{section}
\allowdisplaybreaks

\newenvironment{Ack}%
{\par \vspace{\baselineskip}%
 \noindent \textbf{Acknowledgements.}}%
{\par \vspace{\baselineskip}}

\setenumerate{label=(\arabic*),nosep}
\setitemize{nosep}
\newlist{clist}{enumerate}{1}
\setlist*[clist]{label=(\roman*), nosep}

\theoremstyle{definition}
\newtheorem{thm}{Theorem}[section]
\newtheorem{dfn}[thm]{Definition}
\newtheorem{prp}[thm]{Proposition}
\newtheorem{lem}[thm]{Lemma}
\newtheorem{cor}[thm]{Corollary}

\newtheorem{rmk}[thm]{Remark}

\newtheorem*{rmk*}{Remark}

\crefname{thm}{Theorem}{Theorems}
\crefname{dfn}{Definition}{Definitions}
\crefname{prp}{Proposition}{Propositions}
\crefname{lem}{Lemma}{Lemmas}
\crefname{cor}{Corollary}{Corollaries}
\crefname{rmk}{Remark}{Remarks}
\crefname{eg}{Example}{Examples}
\crefname{figure}{Figure}{Figures}
\crefname{table}{Table}{Tables}
\crefname{section}{\S\!}{\S\S\!}
\crefname{subsection}{\S\!}{\S\S\!}
\crefname{appendix}{Appendix}{Appendices}
\crefname{equation}{}{}

\newcommand{\bl}{\bullet}
\newcommand{\ep}{\epsilon}
\newcommand{\ve}{\varepsilon}
\newcommand{\ol}{\overline}
\newcommand{\ul}{\underline}
\newcommand{\wh}{\widehat}
\newcommand{\wt}{\widetilde}
\newcommand{\ceq}{\coloneqq} 

\newcommand{\xr}{\xrightarrow}
\newcommand{\xrr}[1]{\xrightarrow{\ #1 \ }{}}
\newcommand{\inj}{\hookrightarrow}

\newcommand{\lto}{\longrightarrow}
\newcommand{\sto}{\xr{\sim}}
\newcommand{\lsto}{\xrr{\sim}}
\newcommand{\mto}{\mapsto}
\newcommand{\lmto}{\longmapsto}

\newcommand{\q}{\textup{q}}
\newcommand{\an}{\textup{an}}
\newcommand{\GF}{\textup{GF}}
\newcommand{\hol}{\textup{hol}}
\newcommand{\std}{\textup{std}}
\newcommand{\tcl}{\textup{cl}}
\newcommand{\BHS}{\textup{BHS}}
\newcommand{\MSV}{\textup{MSV}}
\newcommand{\sing}{\textup{sing}}

\newcommand{\tred}{\textup{red}}
\newcommand{\SUSY}{\textup{SUSY}}
\newcommand{\Weiss}{\textup{Weiss}}

\newcommand{\bbC}{\mathbb{C}}
\newcommand{\bbR}{\mathbb{R}}
\newcommand{\bbZ}{\mathbb{Z}}
\newcommand{\sC}{\mathbb{C}^{1|1}}
\newcommand{\sD}{\mathbb{D}^{1|1}}
\newcommand{\clA}{\mathcal{A}}
\newcommand{\clD}{\mathcal{D}}
\newcommand{\clE}{\mathcal{E}}
\newcommand{\clF}{\mathcal{F}}
\newcommand{\clH}{\mathcal{H}}
\newcommand{\clJ}{\mathcal{J}}
\newcommand{\clM}{\mathcal{M}}
\newcommand{\clV}{\mathcal{V}}
\newcommand{\sfC}{\mathsf{C}}
\newcommand{\Ch}{\mathsf{Ch}}
\newcommand{\DVS}{\mathsf{DVS}}
\newcommand{\shO}{\EuScript{O}}
\newcommand{\frg}{\mathfrak{g}}
\newcommand{\frt}{\mathfrak{t}}

\newcommand{\ev}{\ol{0}}
\newcommand{\od}{\ol{1}}
\newcommand{\Zt}{\bbZ/2\bbZ}
\newcommand{\pd}{\partial}
\newcommand{\bpd}{\overline{\partial}}
\newcommand{\sSig}{\Sigma^{1|1}}
\newcommand{\hf}{\frac{1}{2}}
\newcommand{\thf}{\tfrac{1}{2}}
\newcommand{\vac}{\bm{1}}

\newcommand{\abs}[1]{\left| #1 \right|} 
\newcommand{\dbr}[1]{\llbracket #1 \rrbracket} 
\newcommand{\pair}[1]{\langle #1 \rangle}
\newcommand{\cObs}[1]{\mathop{\Obs^{\textup{cl}}_{#1}}}
\newcommand{\qObs}[1]{\mathop{\Obs^{\textup{q}}_{#1}}}
\newcommand{\scObs}[1]{\mathop{\EuScript{O}bs^{\tcl}_{#1}}}
\newcommand{\sqObs}[1]{\mathop{\EuScript{O}bs^{\q}_{#1}}}
\newcommand{\rst}[2]{\left. #1 \right|_{#2}}

\newcommand{\dbc}{%
 \mathrel{\vcenter{\hbox{\ooalign{%
  \raisebox{0.5ex}{$\scriptstyle \circ$}\cr%
  \raisebox{-0.5ex}{$\scriptstyle \circ$}%
 }}}}%
}
\newcommand{\nod}[1]{\mathopen{\dbc} \mathrel{#1} \mathclose{\dbc}} 

\DeclareMathOperator{\id}{id}

\DeclareMathOperator{\evl}{ev}
\DeclareMathOperator{\CdR}{CdR}

\DeclareMathOperator{\bbD}{\mathbb{D}}
\DeclareMathOperator{\bbV}{\mathbb{V}}
\DeclareMathOperator{\Ber}{Ber}
\DeclareMathOperator{\Der}{Der}
\DeclareMathOperator{\Hom}{Hom}
\DeclareMathOperator{\End}{End}

\DeclareMathOperator{\Map}{Map}
\DeclareMathOperator{\Ric}{Ric}
\DeclareMathOperator{\Sym}{Sym}
\DeclareMathOperator{\Obs}{Obs}
\DeclareMathOperator{\CAlg}{CAlg}

\DeclareMathOperator{\Disk}{\mathsf{Disks}}
\DeclareMathOperator*{\sres}{sres}
\DeclareMathOperator{\Wedge}{{\textstyle\bigwedge}}
\DeclareMathOperator{\hocolim}{hocolim}

\begin{document}

\title{BV construction of SUSY vertex algebras from SUSY factorization algebras}
\author{Shintarou Yanagida}
\address{Graduate School of Mathematics, Nagoya University.
 Furocho, Chikusaku, Nagoya, Japan, 464-8602.}
\email{yanagida@math.nagoya-u.ac.jp}
\date{June 4, 2026}
\keywords{Costello--Gwilliam factorization algebras, supersymmetric vertex algebras}

\begin{abstract}
We construct $N=1$ supersymmetric (SUSY) vertex algebras from 
supersymmetric enhancements of Costello--Gwilliam factorization algebras 
on super Riemann surfaces.
Introducing SUSY factorization algebras defined on embedded SUSY disks 
together with natural symmetry conditions,
we prove a SUSY analogue of the Costello--Gwilliam extraction theorem.
As an application, we study the holomorphic sigma model in the BV formalism.
For a linear target, we obtain the free $bc$-$\beta\gamma$ system 
and recover its structure as a SUSY vertex algebra. 
For general complex targets, we describe the descent of the theory 
under coordinate changes and identify the resulting SUSY vertex algebra 
with the chiral de Rham complex. 
We further show that Ricci-flat K\"ahler and hyperk\"ahler targets 
give rise to $N=2$ and $N=4$ supersymmetric enhancements 
introduced by Ben-Zvi--Heluani--Szczesny.
\end{abstract}

\maketitle
{\small \tableofcontents}

\setcounter{section}{-1}
\section{Introduction}

We aim to formulate \emph{holomorphic factorization algebras} 
on a \emph{SUSY curve} (a superconformal curve, or a super Riemann surface)
in the Costello--Gwilliam sense \cite{CG1,CG2}, 
and to use them to construct \emph{SUSY vertex algebras} \cite{HK}
via local-to-global methods. 
The motivating example is the $bc$-$\beta\gamma$ system 
whose local operator algebra matches the free-field description 
underlying the chiral de Rham complex (CdR) 
introduced by Malikov--Schechtman--Vaintrob \cite{MSV},
together with the geometric $N{=}1,2,4$ SUSY structures constructed by
Ben-Zvi--Heluani--Szczesny \cite{BHS}.

The purpose of this paper is to establish a systematic link between 
\emph{SUSY factorization algebras} and \emph{SUSY vertex algebras}, 
extending the framework of Costello--Gwilliam \cite{CG1,CG2} 
to the setting of super-geometry.

Factorization algebras provide a powerful local-to-global formalism
for quantum field theory, 
encoding observables as a cosheaf-like structure on space-time. 
In this paper, we focus on two-dimensional chiral conformal field theories 
and their supersymmetric extensions. 
The corresponding theories of factorization algebras 
have been developed by several authors, including the works by 
Beilinson--Drinfeld \cite{BD},
Kapranov--Vasserot \cite{KV1,KV2,KV3,KV4,KV5}, 
Francis--Gaitsgory \cite{FG}, 
and Costello--Gwilliam \cite{CG1,CG2}.
The present work follows the Costello--Gwilliam approach.
We refer the reader to \cite{CG,GW} for recent developments in this direction.

In the holomorphic setting, 
Costello--Gwilliam \cite[Chapter 5]{CG1} showed that suitable
factorization algebras on the complex plane give rise to vertex algebras. 
Since then, there have been numerous studies exploring the relationship 
between Costello--Gwilliam factorization algebras and vertex algebras, 
including works by Gwilliam \cite{G}, Williams \cite{Wil}, 
Bruegmann \cite{Br1,Br2}, Vicedo \cite{V}, and Nishinaka \cite{N1,N2}.

\medskip 

The goal of this paper is to extend this correspondence to supersymmetric 
theories by replacing the complex plane with an $N=1$ supercurve 
and vertex algebras with SUSY vertex algebras 
in the sense of Heluani--Kac \cite{HK}. 
The latter encode the algebraic structure of observables in two-dimensional
conformal field theories enhanced by supersymmetry.

There have been numerous studies on SUSY vertex algebras, 
including works by Heluani \cite{H1,H2}, Ben-Zvi--Heluani--Szczesny \cite{BHS},
S.Y.\ \cite{Y}, and Nishinaka--S.Y.\ \cite{NY}. 
The present work is a continuation of the latter two papers, and 
aims to clarify their relationship with Costello--Gwilliam factorization algebras.

\medskip 

To achieve this goal, in \cref{s:1} we introduce the notion of an $N=1$ 
SUSY prefactorization algebra defined on a category of embedded SUSY disks. 
This replaces the usual category of disks in a Riemann surface 
and allows one to incorporate odd directions. 
Imposing Weiss descent yields the notion of a SUSY factorization algebra.
We then formulate natural conditions reflecting holomorphicity, 
super-translation invariance, and $S^1$-equivariance, 
and prove a SUSY analogue (\cref{thm:1:533}) 
of the Costello--Gwilliam construction \cite[\S5.3]{CG1}: 
under these assumptions, the cohomology of observables 
on a small SUSY disk carries a canonical structure of 
an $N_K=1$ SUSY vertex algebra in the sense of \cite[\S4]{HK}. 
This can be viewed as an ``extraction functor'' from 
SUSY factorization algebras to SUSY vertex algebras,
a SUSY extension of the construction in \cite[Chapter 5]{CG1}.
We also introduce the notion of SUSY Poisson factorization algebras,
which will be used in the next \cref{s:2}.

The main source of examples comes from the holomorphic sigma model
with $N=1$ supersymmetry, as we will show in \cref{s:2}. 
We formulate this theory in the Batalin--Vilkovisky (BV) 
formalism using superfields, 
a SUSY extension of the construction in \cite{CG2}. 
For a linear target, the model reduces to the free $bc$-$\beta\gamma$ system.
We show in \cref{thm:2:V(cObs)} that the classical observables 
form a SUSY Poisson factorization algebra,
and in \cref{thm:2:qObsVA} that its quantization produces 
a SUSY vertex algebra whose operator product expansion (OPE)
coincides with the standard free-field realization 
underlying the chiral de Rham complex \cite{MSV}.

In \cref{s:3}, for a general complex manifold $X$, 
we construct local factorization algebras modeled on 
the free theory in coordinate charts 
and study their descent under holomorphic coordinate changes. 
We show in \cref{thm:3:obs} that the resulting sheaf of SUSY vertex algebras
coincides with the chiral de Rham complex of $X$. 
In this way, we reinterpret the construction of 
Malikov--Schechtman--Vaintrob \cite{MSV} and 
its supersymmetric refinement by Ben-Zvi--Heluani--Szczesny \cite{BHS}
from the viewpoint of factorization algebras.

Finally, in \cref{s:4}, we incorporate geometric structures on the target.
When $X$ is Ricci-flat K\"ahler or hyperk\"ahler, 
the BV theory admits additional local currents whose extracted fields
generate $N=2$ and $N=4$ SUSY vertex algebras, respectively
(\cref{thm:4:N=2,thm:4:N=4}). 
This lifts the structures on the chiral de Rham complex in \cite{BHS} 
to the level of SUSY factorization algebras.

\subsection*{Organization}

The paper is organized as follows. 
In \cref{s:1}, we develop the general theory of 
SUSY factorization algebras and prove the extraction theorem. 
In \cref{s:2}, we analyze the linear sigma model 
and identify the resulting SUSY vertex algebra. 
In \cref{s:3}, we globalize the construction to arbitrary complex manifolds.
In \cref{s:4}, we study enhanced supersymmetry 
in the presence of special geometric structures.

\subsection*{Global notation and terminology}

We summarize here the notation and terminology used throughout the main text.
\begin{itemize}
\item
The symbols $\bbZ$, $\bbR$, and $\bbC$ denote 
the set of all integers, real numbers, and complex numbers, respectively.

\item
Unless otherwise specified, 
linear spaces are defined over the field $\bbC$, 
and $\Hom$, $\End$, and $\otimes$ are understood in this category

\item
For a linear space $V$, we denote the linear dual $\Hom(V,\bbC)$ by $V^\vee$.

\item 
In \cref{s:2}--\cref{s:4}, we work in the dg super setting.
We denote the cohomological degree by $\abs{\cdot} \in \bbZ$,
and the parity (the super degree) by $p(\cdot)\in\Zt=\{\ev,\od\}$.

\item
Einstein's summation convention is used throughout. 
For example, $x^ip_i \ceq \sum_i x^i p_i$.

\item
We follow \cite{FB,K} for the terminology of vertex algebras.
\end{itemize}

\section{Constructing SUSY vertex algebras from SUSY factorization algebras}
\label{s:1}

In this section we establish a supersymmetric (SUSY) enhancement of 
the construction \cite[Chap.\ 5]{CG1}
of vertex algebras from Costello--Gwilliam (CG) factorization algebras
on $\bbC$ satisfying several conditions. 
To achieve this, we start with introducing a correct SUSY counterpart 
of CG factorization algebras.
To motivate it, recall that the non-super argument \cite[\S3.1, \S5.1]{CG1} 
uses the (multi)category $\Disk(\bbC)$ 
of open disks in $\bbC$ and their embeddings, 
but the resulting CG factorization algebra does not see the supergeometry.
In order to take ``odd directions'' into consideration, we work on 
a basis of \emph{embedded SUSY disks} in an \emph{$N=1$ SUSY curve}.

\subsection{SUSY curves and SUSY prefactorization algebras}
\label{ss:1.1}

Let $X$ be an $N=1$ SUSY curve in the sense of \cite[Chap.\ 2]{M}, 
i.e., a complex supermanifold of dimension $1|1$ 
equipped with a rank $0|1$ distribution $\clD_X \subset T_X$ 
that is maximally non-integrable. 
It is equivalent to the notion of an $N=1$ superconformal curve 
in the sense of \cite{BR}.
Locally there exist super-coordinates\footnote{The standard notation is
 $(z,\theta)$, but we use $\zeta$ for the odd coordinate 
 and reserve the symbol $\theta$.} 
$(z,\zeta)$ for which $\clD_X$ is generated by the odd vector field
\begin{align}\label{eq:1:D}
 D_Z \ceq \pd_{\zeta} + \zeta\,\pd_z.
\end{align}
The standard local model is the $N=1$ SUSY disk $\sD$,
which is the complex superdisk of dimension $1|1$ 
with the distribution $\clD_{\std}$ generated by this $D_Z$.

As a SUSY analogue of the aforementioned category $\Disk(\Sigma)$ 
for a Riemann surface $\Sigma$, we consider:

\begin{dfn}[{The category of SUSY disks}]
Define a symmetric monoidal category $\Disk^{\SUSY}(X)$ by: 
\begin{itemize}
\item 
Objects are pairs $(U,\varphi)$ where $U\subset X$ is an open sub-supermanifold
admitting a finite decomposition $U=\bigsqcup_{i=1}^n U_i$ and
$\varphi = (\varphi_i\colon \sD \sto U_i)_{i=1}^n$ 
are superconformal isomorphisms 
(i.e.,\ $(\varphi_i)_*(\clD_{\std})=\rst{\clD_X}{U_i}$).

\item 
Morphisms $(U,\varphi) \to (V,\psi)$ are 
inclusions of opens $U \inj V$ viewed over $X$.
Compositions of morphisms and identities are naturally defined.

\item
The symmetric monoidal structure is given by disjoint union, 
$(U,\varphi) \otimes (V,\psi) \ceq (U \sqcup V,\varphi\sqcup\psi)$ 
when $U\cap V=\emptyset$, with unit $\emptyset$.
\end{itemize}
In the following, we denote an object $(U,\varphi)$ by $U$ for simplicity.
\end{dfn}

%
%
%

Then, the following is a natural SUSY analogue of 
a CG prefactorization algebra on a Riemann surface \cite[\S3.1, \S5.1]{CG1}.


\begin{dfn}
Let $(\sfC,\otimes,\bm{1})$ a symmetric monoidal ($\infty$-)category 
admitting (homotopy) colimits. 
An \emph{$N=1$ SUSY prefactorization algebra} on an $N=1$ SUSY curve $X$ 
with values in $\sfC$ is a symmetric lax monoidal functor
\[
 \clF\colon \Disk^{\SUSY}(X) \lto \sfC.
\]
In other words, it is a functor together with a morphism 
$\eta\colon \bm{1} \to \clF(\emptyset)$ and morphisms 
$\mu_{U \sqcup V}\colon \clF(U) \otimes \clF(V) \to \clF(U \sqcup V)$
for all $U,V \in \Disk^{\SUSY}(X)$,
satisfying the standard associativity, naturality and unitality conditions.
\end{dfn}

In particular, for a disjoint union 
$U=\bigsqcup_{i=1}^n U_i \in \Disk^{\SUSY}(X)$, we have a morphism
\begin{align}\label{eq:1:str}
 \mu_U\colon \clF(U_1) \otimes \dotsb \otimes \clF(U_n) \lto 
 \clF\Bigl(\bigsqcup_{i=1}^n U_i\Bigr), 
\end{align}
which we call the \emph{factorization product} for $U$.
In addition, for an inclusion $j\colon U \inj V$, we have a morphism
\[
 \clF(j)\colon \clF(U) \lto \clF(V),
\]
which we call the \emph{extension map} for $j$.

For later use, let us introduce 
a straightforward analogue of \cite[Definition 3.1.3]{CG1}: 

\begin{dfn}\label[dfn]{dfn:1:unital}
An $N=1$ SUSY prefactorization algebra $\clF$ valued in $(\sfC,\otimes,\bm{1})$
is called \emph{unital} 
if it has a morphism $\bm{1} \to \clF(\emptyset)$ in $\sfC$, or equivalently, 
if the commutative algebra object $\clF(\emptyset)$ in $\sfC$ is unital.
\end{dfn}


\subsection{SUSY factorization algebras} 

Following the argument in \cite{CG1}, 
we define a \emph{SUSY factorization algebra} to be 
a SUSY prefactorization algebra which satisfies 
the \emph{Weiss} (\emph{cosheaf}) \emph{descent} condition 
after extension from the superconformal disk basis to the Weiss site of opens.
Let us make this precise as follows.

\begin{dfn}
Let $X$ be an $N=1$ SUSY curve, and let $U\in \Disk^{\SUSY}(X)$.
A family of morphisms (inclusions)
$\{j_\alpha\colon U_\alpha \inj U\}_{\alpha \in A}$
is called a \emph{Weiss cover} of $U$ 
if the following holds on the reduced spaces:
for every finite subset $S$ of the underlying topological space
$\abs{U_{\tred}}$ of the reduced manifold $U_{\tred}$, 
there exists $\alpha\in A$ such that $S \subset \abs{(U_\alpha)_{\tred}}$.
\end{dfn}

We denote by $J_{\Weiss}^{\SUSY}$ the Grothendieck topology 
on $\Disk^{\SUSY}(X)$ generated by the Weiss covering families. 
In other words, a sieve $R$ on $U$ is declared to be a covering if it contains
the sieve generated by some Weiss cover $\{U_\alpha \inj U\}_{\alpha\in A}$.
The resulting site is denoted by $(\Disk^{\SUSY}(X),J_{\Weiss}^{\SUSY})$.

\begin{dfn}
\label[dfn]{dfn:1:SFA}
An \emph{$N=1$ SUSY factorization algebra} on $X$ valued in a symmetric monoidal
($\infty$-)category $\sfC$ is a SUSY prefactorization algebra $\clF$ 
such that the underlying covariant functor
$\clF\colon \Disk^{\SUSY}(X) \to \sfC$
is a cosheaf for the Grothendieck topology $J^{\SUSY}_{\Weiss}$.
\end{dfn}

Concretely, for every object $U\in\Disk^{\SUSY}(X)$ and every Weiss cover
$\{U_\alpha \inj U\}_{\alpha\in A}$, the canonical map
from the (homotopy) \v{C}ech colimit to $\clF(U)$ is an equivalence in $\sfC$:
\[
 \hocolim\Bigl( 
  \coprod_{\alpha,\beta}\clF(U_\alpha\cap U_\beta) \rightrightarrows
  \coprod_{\alpha}\clF(U_\alpha) \Bigr)
 \lsto \clF(U).
\]

\subsection{$S^1$-equivariance and holomorphically super-translation-invariance}

The construction in \cite[Chap.\ 5]{CG1} is applied to 
a prefactorization algebra on $\bbC$ satisfying several conditions.
To make a SUSY analogue, let us begin with:

\begin{dfn}
\label[dfn]{dfn:1:Der}
Let $\clF$ be an $N=1$ SUSY prefactorization algebra on an $N=1$ SUSY curve $X$
with values in a symmetric monoidal dg category $\sfC$. 
A \emph{derivation} of (cohomological) degree $k$ is a natural transformation
$\delta\colon \clF \to \clF[k]$ such that 
for every $U=\bigsqcup_{i=1}^n U_i \in \Disk^{\SUSY}(X)$ 
there is a graded Leibniz rule with respect to the factorization products:
\[
 \delta \circ \mu_{U} = \sum_{i=1}^n  (-1)^{k(\deg x_1+\cdots+\deg x_{i-1})}\,
 \mu_U \circ (\id\otimes\cdots\otimes \delta \otimes\cdots\otimes \id),
\]
where 
$\mu_U \colon \clF(U_1)\otimes\cdots\otimes\clF(U_n) \to \clF(\bigsqcup_i U_i)$
is the structure map \eqref{eq:1:str}.
We write $\Der(\clF)$ for the dg Lie superalgebra of derivations.
\end{dfn}

Hence, a morphism $\frg \to \Der(\clF)$ of dg Lie superalgebras 
encodes the $\frg$-symmetry of $\clF$.

Next, we introduce a dg Lie superalgebra 
encoding the $N=1$ superconformal geometry.
Hereafter we take $X=\sC$ with coordinates $(z,\zeta)$.
We consider odd vector fields 
\[
 D \ceq \pd_{\zeta}+\zeta\pd_z. \quad 
 E \ceq z \pd_z + \tfrac{1}{2} \zeta \pd_\zeta,
\]
where $D$ is from \eqref{eq:1:D}, 
and $E$ is the infinitesimal generator (Euler operator)
of the $S^1$-action on $\sC$, 
$e^{i\phi}\cdot (z,\zeta) \ceq (e^{i\phi}z,\ e^{i\phi/2}\zeta)$.
Then, using the super-commutator $[\cdot,\cdot]$ of vector fields, we have 
\begin{align}\label{eq:1:DE}
  [D,D]=2\pd_z, \quad [E,\pd_z]=-\pd_z,\quad [E,D]=-\tfrac{1}{2} D.
\end{align}

\begin{dfn}
Let $\frt^{\SUSY}_{\mathrm{hol},S^1}$ be the dg Lie superalgebra generated by:
\begin{itemize}
\item 
even degree $0$ elements $\pd_z$, $\pd_{\ol{z}}$, $E$, $\ol{E}$;
\item 
odd degree $0$ elements $D$, $\ol{D}$;
\item 
even degree $-1$ element $\eta_{\ol{z}}$ and 
odd degree $-1$ element $\eta_{\ol{D}}$;
\item 
even degree $-1$ element $\eta_{\ol{E}}$;
\end{itemize}
with differential determined by
\[
 d(\eta_{\ol{z}})=\pd_{\ol{z}}, \quad 
 d(\eta_{\ol{D}})=\ol{D}, \quad 
 d(\eta_{\ol{E}})=\ol{E},
\]
and brackets extending those of the corresponding vector fields, 
in particular \eqref{eq:1:DE} and similarly for the barred generators.
\end{dfn}

Now, the following SUSY analogue makes sense. 

\begin{dfn}\label[dfn]{dfn:1:SETI}
Let $\clF$ be a SUSY prefactorization algebra on $\sC$ 
valued in a symmetric monoidal dg category $\sfC$. 
We say that $\clF$ is 
\emph{$S^1$-equivariant and holomorphically super-translation invariant} if
there is an action of the dg Lie superalgebra 
$\frt^{\SUSY}_{\hol,S^1}$ by derivations:
\[
 \rho\colon \frt^{\SUSY}_{\hol,S^1} \lto \Der(\clF),
\]
where $\Der(\clF)$ denotes the derivations of $\clF$ (\cref{dfn:1:Der}).

%
%
\end{dfn}

In particular, setting $T,S\in\Der^0(\clF)$ and 
$h_{\ol{z}},h_{\ol{D}},h_{\ol{E}}\in\Der^{-1}(\clF)$ by 
\begin{align}\label{eq:1:TS}
 T \ceq \rho(\pd_z), \quad 
 S \ceq \rho(D), \quad 
 h_{\ol{z}} \ceq \rho(\eta_{\ol{z}}), \quad 
 h_{\ol{D}} \ceq \rho(\eta_{\ol{D}}), \quad 
 h_{\ol{E}} \ceq \rho(\eta_{\ol{E}}),
\end{align}
the degree 0 derivations satisfy the same brackets as \eqref{eq:1:DE}: 
\[
 [S      ,S]=2T, \quad 
 [\rho(E),T]=-T, \quad 
 [\rho(E),S]=-\tfrac{1}{2}S, 
\]
and the anti-holomorphic generators act homotopically trivially:
\[
 [d_{\clF},h_{\ol{z}}]=\rho(\pd_{\ol{z}}),\quad
 [d_{\clF},\rho(\eta_{\ol{D}})]=\rho(\ol{D}),\quad
 [d_{\clF},\rho(\eta_{\ol{E}})]=\rho(\ol{E}),
\]
where $d_{\clF}$ is the internal differential of $\clF$.

\subsection{SUSY analogue of the Costello--Gwilliam construction}
\label{ss:1.4}

Here we give a natural $N=1$ SUSY analogue of the construction of vertex algebras
from factorization algebras by Costello and Gwilliam \cite[Theorem 5.3.3]{CG1}.

We begin with the value of prefactorization algebras.
Following \cite{CG1}, we take
\[
 \sfC \ceq \Ch(\DVS),
\]
the category of cochain complexes in the symmetric monoidal category $\DVS$ of
differentiable vector spaces \cite[Appendix B]{CG1}.
We call the $\bbZ$-grading on $\sfC$ the cohomological degree.
It has the symmetric monoidal structure $\otimes$, 
which is the degree-wise extension of the completed tensor product in $\DVS$.
The unit object is denoted by $\bm{1}$.
The category $\sfC$ also has a model structure, 
where weak equivalences are quasi-isomorphisms.

For $r \in \bbR_{>0}$, we denote by $\bbD_r$ 
the open disk in $\bbC$ with center $0$ and radius $r$,
and by $\sD_r$ the $N=1$ SUSY disk with underlying space $\sD_r$
and the distribution generated by $D_Z=\pd_\zeta+\zeta\pd_z$.
We also denote by $\sD_r(Z)$ the SUSY disk with center $Z=(z,\zeta)$.
More explicitly, using the super-translation $\tau^{\SUSY}_Z$, we define 
\[
 \sD_r(Z) \ceq \tau^{\SUSY}_{Z}\bigl(\sD_r\bigr), \quad 
 \tau^{\SUSY}_{(z,\zeta)}(w,\omega) \ceq (w+z+\zeta\omega,\omega+\zeta),
\]
noting that $\tau^{\SUSY}_Z$ preserves $D_Z$ (pushes $D_Z$ to itself).

Let $\clF$ be an $N=1$ SUSY prefactorization algebra on $\sC$ 
valued in $\sfC=\Ch(\DVS)$.
We consider the following conditions.
\begin{clist}
\item \label{i:1:u}
$\clF$ is unital (\cref{dfn:1:unital}).
We denote the unit map by $\bm{1} \to \clF(\emptyset)$.

\item \label{i:1:SETI}
$\clF$ is $S^1$-equivariant and holomorphically super-translation invariant 
(\cref{dfn:1:SETI}).
We use the symbols $T\ceq\rho(\pd_z)$ and $S\ceq\rho(D)$ from \eqref{eq:1:TS}.

\item \label{i:1:di}
(Disk independence) 
For any $s \ge r > 0$, the extension maps $\clF(\sD_r) \to \clF(\sD_s)$ 
induced by the inclusion $\sD_r \inj \sD_s$ are quasi-isomorphisms in $\sfC$.

\item \label{i:1:rho(E)}
(Weight boundedness)
The cohomology $\bbV \ceq H^{\bl}\!\bigl(\clF(\sD_1)\bigr)$
is a locally finite, lower-bounded $\thf\bbZ$-graded module for 
the infinitesimal super-rotation operator $\rho(E)$.

\item \label{i:1:hol}
(Holomorphic dependence) 
Fix radii $0<r\ll R$. 
For each $n \ge 1$ and each configuration $\ul{x}=(Z_1;\dots;Z_n)$
such that the embedded SUSY disks $\sD_r(Z_i) \subset \sD_R$ 
are pairwise disjoint, let
\begin{align}\label{eq:1:mrRn}
 m_{r,R,n}(\ul{x})\colon \clF\bigl(\sD_r\bigr)^{\otimes n} \lto 
 \bigotimes_{i=1}^n \clF\bigl(\sD_r(Z_i)\bigr) \xrr{\mu} 
 \clF\left(\bigsqcup_{i=1}^n \sD_r(Z_i)\right)
 \lto \clF\bigl(\sD_R\bigr),
\end{align}
where 
the first arrow is the structure map obtained by 
super-translating $\sD_r$ to each $Z_i$, 
$\mu$ denotes the factorization product, 
and the last arrow uses functoriality for the inclusion 
$\bigsqcup_{i=1}^n\sD_r(Z_i) \inj \sD_R$.
Then, we require:
\begin{enumerate}
\item[(a)] 
$m_{r,R,n}(\ul{x})$ depends smoothly on the reduced parameters $(z_i,\ol{z}_i)$
(in the $\DVS$ sense) and polynomially on the odd parameters $\zeta_i$;

\item[(b)]
(\emph{Cauchy--Riemann up to homotopy}) 
for each $i=1,\dotsc,n$ there exists a multilinear map
\[
 h^{(i)}_{r,R,n}(\ul{x})\colon
 \clF\bigl(\sD_r(0)\bigr)^{\otimes n} \lto 
 \clF\bigl(\sD_R(0)\bigr)
\]
of degree $-1$, depending smoothly on $\underline{x}$ such that
\[
 \pd_{\ol{z}_i} m_{r,R,n}(\ul{x}) =
 d_{\clF} \circ h^{(i)}_{r,R,n}(\ul{x}) + 
 h^{(i)}_{r,R,n}(\ul{x}) \circ d_{\clF},
\]
where $d_{\clF}$ denotes the differentials on the source and target complexes.
\end{enumerate}
Consequently, the induced operations on cohomology 
are holomorphic in the variables $z_i$.
\end{clist}

\begin{rmk}
If we further assume that $\rho(E-\ol{E})$ exponentiates to an $S^1$-action 
by automorphisms on the underlying prefactorization algebra of $\clF$, 
then the condition \ref{i:1:rho(E)} can be replaced by:
The $S^1$-action on $V:=\clF(\bbD_1^{1|1})$ yields 
a decomposition into generalized eigenspaces (weights) 
bounded below and compatible with the differential.
\end{rmk}

Recall the axiom of an $N_K=1$ SUSY vertex algebra 
from \cite[Definition 4.10]{HK}.
Now, we have:

\begin{thm}[{SUSY analogue of \cite[Theorem 5.3.3]{CG1}}]
\label{thm:1:533}
Let $\clF$ satisfy the conditions \ref{i:1:u}--\ref{i:1:hol}.
Then, 
\[
 \bbV \ceq H^{\bl}\!\bigl(\clF(\bbD_1^{1|1})\bigr)
\]
carries a structure of an $N=1$ SUSY vertex algebra.
More precisely:
\begin{enumerate}
\item \label{i:thm:1:533:parity}
The parity is given by the cohomology grading on $\bbV$ modulo $2$.

\item 
The vacuum $\bm{1}\in \bbV$ is induced by 
the unit map $\bm{1}\to \clF(\emptyset)$ in the condition \ref{i:1:u} 
followed by functoriality $\clF(\emptyset)\to \clF(\bbD_1^{1|1})$.

\item 
The even and odd translation operators on $\bbV$ are induced by 
$T=\rho(\pd_z)$ and $S=\rho(D)$ in the condition \ref{i:1:SETI}, 
and satisfy $S^2=T$ on cohomology.

\item (state-superfield correspondence) 
For each $a\in\bbV$ there is a \emph{superfield} 
\begin{align}\label{eq:1:YaZ}
 Y(a;Z) = Y(a;z,\zeta) \in \End(\bbV)\dbr{z^{\pm1}}[\zeta]
\end{align}
obtained from the factorization products for two small SUSY disks 
inside a larger SUSY disk (see Proof below for the detail).
Moreover, the correspondence $a \mto Y(a;Z)$ is parity-preserving.

\item (SUSY locality) 
For any $a,b\in\bbV$ there exists $N \in \bbZ_{\ge0}$ such that
\begin{align}\label{eq:1:Sloc}
  z_{12}^N \bigl[Y(a;Z_1),Y(b;Z_2)\bigr] = 0,
\end{align}
where $z_{12} \ceq z_1-z_2-\zeta_1\zeta_2$ and 
$[\cdot,\cdot]$ denotes the super-commutator.
\end{enumerate}
Moreover, the assignment $\clF \mto \bbV$ is functorial,
which will be called the \emph{extraction functor}. 
\end{thm}

\begin{proof}
The argument is a straightforward SUSY analogue of 
the original \cite[Theorem 5.3.3]{CG1}, 
and we only show the outline and the necessary modification.

Fix radii $0 < r \ll 1 < R$, and let $\sD_r$ and $\sD_r(Z)$ denote 
small SUSY disks around $0$ and around $Z=(z,\zeta)$ inside $\bbD_R^{1|1}$.
Fix $a,b \in \bbV$, and choose cocycle representatives 
$\wt{a},\wt{b} \in \clF(\sD_r)$ (using disk independence \ref{i:1:di}).
For $Z=(z,\zeta)$ such that $\sD_r$ and $\sD_r(Z)$ are disjoint in $\sD_R$,
consider the composite $m_{r,R,2}(Z;0)$ 
given in \eqref{eq:1:mrRn}.
Then, we define $Y(a;Z)b \in \bbV$ by  
\begin{align}\label{eq:1:YaZb}
 Y(a;Z)b \ceq \bigl[m_{r,R,2}(Z;0)(\wt{a} \otimes \wt{b})\bigr]
 \in H^{\bl}\!(\sD_R) \simeq \bbV.
\end{align}
where we used disk independence \ref{i:1:di}.

We can verify the vacuum axiom easily as in \cite[\S5.3.4, p.166]{CG1}.
For the SUSY translation axiom, note that the condition \ref{i:1:SETI} 
implies that $\clF$ is a module over the dg Lie superalgebra of 
holomorphic super-translations (subalgebra of $\frt^{\SUSY}_{\hol,S^1}$), 
and the structure maps $m_{r,R,2}(Z;0)$ are morphisms in that module category.
In particular, for $D_Z=\pd_\zeta+\zeta\pd_z$ and $S=\rho(D_Z)$, we have
\[ 
 D_Z m_{r,R,2}(Z;0)(\wt{a} \otimes \wt{b}) = 
 m_{r,R,2}(Z;0)(\rho(D_z)\wt{a} \otimes \wt{b}).
\]
Taking the cohomology as in \eqref{eq:1:YaZb}, we have $D_ZY(a;Z)=Y(Sa;Z)$.
Then, the other parts of the axiom are easily verified.

Below, we prove in four steps that the factorization axioms imply SUSY locality.

\smallskip\noindent
\textbf{Step 1} (single meromorphic function on configuration space). 
This step is similar as the first half of \cite[Proof of Proposition 5.3.6]{CG1}.
Consider the three-disk insertion map
$m_{r,R,3}$ for SUSY disks at $(0,0)$, $(z_1,\zeta_1)$, $(z_2,\zeta_2)$,
defined in \eqref{eq:1:mrRn}.
Using the prefactorization associativity axiom, 
the two composites corresponding to the two parenthesizations agree on overlaps.
By the condition \ref{i:1:hol} (Cauchy--Riemann up to homotopy), 
for any cohomology classes $a,b$, any $l\in\bbZ$, and any linear functional 
$\lambda$ on $H^{l}(\clF(\bbD_R^{1|1}))$, the matrix coefficient
\begin{align}\label{eq:1:flab}
 f_{\lambda,a,b}(z_1,\zeta_1;z_2,\zeta_2) \ceq 
 \lambda\!\left(
 \bigl[m_{r,R,2}(z_1,\zeta_1;z_2,\zeta_2)(\wt{a} \otimes \wt{b})\bigr]\right)
\end{align}
is holomorphic in $z_1,z_2$ on the reduced configuration space $\{z_1\ne z_2\}$,
and polynomial in $\zeta_1,\zeta_2$.

\smallskip\noindent
\textbf{Step 2} (SUSY holomorphicity).
This step needs SUSY consideration.
Using holomorphic super-translation invariance \ref{i:1:SETI}, 
we may translate the configuration so that $(z_2,\zeta_2)=(0,0)$ 
and the dependence is only on the super-translation invariant differences
\begin{align}\label{eq:1:Z12}
 Z_1 - Z_2 = (z_{12},\zeta_{12}) \ceq 
(z_1-z_2-\zeta_1\zeta_2,\zeta_1-\zeta_2).
\end{align}
Therefore, the functions \eqref{eq:1:flab} are holomorphic 
on $\{z_{12} \ne 0\}$, which implies that the super-commutator 
$[Y(a,Z_1),Y(b;Z_2)]$ can be expanded in terms of $z_{12}$. 

\smallskip\noindent
\textbf{Step 3} (finite pole order from $S^1$-finiteness).
This step is similar as \cite[Lemma 5.3.5]{CG1}. 
By weight boundedness \ref{i:1:rho(E)}, $\bbV$ decomposes into 
finite-dimensional generalized weight spaces bounded below.
$S^1$-equivariance \ref{i:1:SETI} implies that $Y(a;z,\zeta)b$ admits 
a Laurent expansion in $z$ whose coefficients lie in weight spaces. 
Bounded-below weights force 
only finitely many negative powers of $z$ for fixed $a,b$. 
Hence, together with Step 2, 
there exists $N \gg 0$ such that $z_{12}^N$ kills all polar parts.

Note that, at this point, for any $a\in\bbV$, $Y(a;Z)$ admits an expansion
\[
 Y(a;Z) = \sum_{n\in\bbZ} z^{-n-1}a_{(n|0)} \;+\;
    \zeta \sum_{n\in\bbZ} z^{-n-1}a_{(n|1)}, \quad 
 a_{(n|p)} \in \End(\bbV).
\]

\smallskip\noindent
\textbf{Step 4} (locality).
This step is quite similar as \cite[Proposition 5.3.6]{CG1}, 
and we only give an outline.
The super-commutator $[Y(a,Z_1),Y(b;Z_2)]$ compares the two expansions 
corresponding to the regions $|z_1|>|z_2|$ and $|z_2|>|z_1|$. 
Since both are expansions of the same meromorphic function (Steps 1--2), 
their difference is supported on the super-diagonal $\{z_{12} = 0\}$, 
and Step 3 shows it is annihilated by $z_{12}^N$. 
Hence, we have the SUSY locality \eqref{eq:1:Sloc}.
\end{proof}

\begin{rmk}\label[rmk]{rmk:1:DVS}
We take the target category $\sfC$ 
to be cochain complexes in differentiable vector spaces, 
which results in the setting \ref{i:thm:1:533:parity} of parity of $\bbV$.
A variation is to take $\sfC$ to be cochains 
in differentiable vector \emph{superspaces},
which results in a shift of parity of $\bbV$.
We omit the details of this variation, and 
refer the reader to \cite[\S3.3]{N2} for further information.
\end{rmk}

\subsection{SUSY Poisson factorization algebra}

We continue to work on $\sfC \ceq \Ch(\DVS)$.
Let $\CAlg(\sfC)$ denote the category of commutative algebra objects in $\sfC$.
Its object is a commutative dg algebra in $\sfC$, 
and a morphism is a dg algebra map.
 
\begin{dfn}\label[dfn]{dfn:1:comSPFA}
An $N=1$ SUSY prefactorization algebra $\clF$ on $\sC$ (valued in $\sfC$)
is called \emph{commutative} if it values in $\CAlg(\sfC)$.
\end{dfn}

In other words, a commutative $N=1$ SUSY prefactorization algebra $\clF$ 
on $\sC$ is a symmetric lax monoidal functor 
$\clF\colon \Disk^{\SUSY}(\sC) \to \sfC$
such that for every $U\in\Disk^{\SUSY}(\sC)$, 
$\clF(U)$ is a commutative dg algebra structure
and all structure maps are algebra maps.


To obtain a Poisson structure, let us consider:

\begin{dfn}\label[dfn]{dfn:1:SPPFA}
An \emph{$N=1$ SUSY Poisson prefactorization algebra} on $\sC$ 
is a commutative $N=1$ SUSY prefactorization algebra $\clF$ on $\sC$  
together with a dg Lie bracket $\{\cdot,\cdot\}$ such that 
\begin{clist}
\item \label{i:dfn:1:SPPFA:obj}
for each $U\in\Disk^{\SUSY}(\sC)$, $\clF(U)$ is a dg Poisson algebra, and 
\item \label{i:dfn:1:SPPFA:mor}
each extension map $\clF(U) \to \clF(V)$ is a morphism of dg Poisson algebras.
\end{clist}
\end{dfn}

Now we have a Poisson analogue of \cref{thm:1:533}:

\begin{thm}\label{thm:1:SPVA}
Let $\clF$ be a Poisson SUSY factorization algebra on $\sC$ 
with Poisson bracket 
$\{\cdot,\cdot\}\colon \clF(U) \otimes \clF(U)\to \clF(U)$.
Assume that $\clF$ satisfies the same conditions \ref{i:1:u}--\ref{i:1:hol}
used in \cref{thm:1:533} to construct 
the $N=1$ SUSY vertex superalgebra structure on
\[
 \bbV \ceq H^{\bl}\!\bigl(\clF(\sD_1)\bigr).
\]
Then $\bbV$ carries a structure of an $N_K=1$ SUSY Poisson vertex algebra
in the sense of \cite[\S4.10]{HK}, \cite[\S3.1.4]{Y}.
More precisely: 
\begin{enumerate}
\item
Let $Y(a;Z)$ denote the superfield \eqref{eq:1:YaZ} for $a\in\bbV$ 
and $Z=(z,\zeta)$ produced by the non-Poisson construction of \cref{thm:1:533}.
Then, for any $a,b\in\bbV$, we have
\begin{align}\label{eq:1:Yser}
 Y(a;Z)b \in \bbV\dbr{z}[\zeta],
\end{align}
so that $\bbV$ is a commutative SUSY vertex algebra 
in the sense of \cite[\S2.4.1]{Y}.

\item 
Let $\Lambda=(\lambda,\chi)$ with $\lambda$ even and $\chi$ odd, 
subject to the $N_K=1$ relation $\chi^2=-\lambda$, 
so $\bbC[\Lambda]=\bbC[\lambda]\oplus\chi\,\bbC[\lambda]$. 
Define a bilinear map (the SUSY $\Lambda$-bracket)
\[
 \{-_\Lambda-\}\colon \bbV \otimes \bbV \lto \bbV[\Lambda]
\]
by the super-residue of the singular part of the two-point Poisson OPE:
\begin{align}\label{eq:1:P}
 \{a_\Lambda b\} \ceq
 \sres_{Z=0} \Bigl(e^{z\lambda+\zeta\chi}\,
 \bigl(\{\,Y(a;Z),\,b\,\}\bigr)_{\sing}\Bigr),
\end{align}
where $\sres_{Z=0}$ is the super-residue
(or the Berezin residue, the coefficient of $z^{-1}\zeta$), 
and $(-)_{\sing}$ denotes the singular part in $z$.
Then $(\bbV,\{\cdot_\Lambda \cdot\})$ satisfies 
the $N_K=1$ SUSY Lie conformal identities \cite[Definition 4.10]{HK}, 
and $\{\cdot_\Lambda \cdot\}$ is a derivation with respect to 
the super-commutative product $\cdot$. 
\end{enumerate}
\end{thm}

At the last part, we used the equivalence \cite[\S3.1.4]{Y} 
between a commutative SUSY vertex algebra $(\bbV,\bm{1},S,Y)$ 
and a unital commutative superalgebra $(\bbV,\cdot,\bm{1})$ 
with odd derivation $S$.

\begin{proof}
First, we show the commutativity of $\bbV$.
Fix radii $0<r\ll 1<R$ and identify
\[
 \bbV = H^{\bl}\!\bigl(\clF(\sD_1)\bigr)\cong H^{\bl}\!\bigl(\clF(\sD_R)\bigr)
\]
using the disk independence \ref{i:1:di}.
For $Z=(z,\zeta)$ with $\sD_r$ and $\sD_r(Z)$ 
disjoint inside $\sD_R$, let
\[
 m(Z) \ceq m_{r,R,2}(Z;0)\colon 
 \clF(\sD_r) \otimes \clF(\sD_r) \lto \clF(\sD_R)
\]
denote the composition \eqref{eq:1:mrRn}, i.e., 
the factorization map obtained by inserting the first copy as $\sD_r$ 
and the second as $\sD_r(Z)$, and then extending to $\sD_R$.
By the holomorphic dependence assumption, 
$m(z,\zeta)$ depends holomorphically on $z$ up to homotopy, 
and polynomially on $\zeta$ on the domain $0<\abs{z}<\ve$, 
where the disks remain disjoint.

Since $\clF$ is a commutative factorization algebra, 
each $\clF(U)$ is a commutative dg algebra and
all structure maps are morphisms of commutative dg algebras. 
In particular, for any pairwise disjoint inclusions $U_1 \sqcup U_2 \inj V$,
the composition of the factorization product and the extension map
\[
 \clF(U_1) \otimes \clF(U_2) \xrr{\mu_{U_1 \sqcup U_2}} 
 \clF(U_1\sqcup U_2) \lto \clF(V)
\]
coincides with the multiplication $\cdot$ in $\clF(V)$ applied to 
the images of the extension maps $\iota_i\colon \clF(U_i) \inj \clF(V)$.
Applying this to $U_1=\sD_r(Z)=\sD_r(z,\zeta)$ and $U_2=\sD_r$ inside $V=\sD_R$
shows that, for $0<|z|<\ve$,
\[
 m(z,\zeta)(x \otimes y) = 
 \bigl(\iota_1(x)\bigr) \cdot \bigl(\iota_2(y)\bigr) \in \clF(\sD_R).
\]
Now, note that the right-hand side makes sense \emph{also at} $z=0$
since it only uses the multiplication $\iota$ in $\clF(\sD_R)$ 
and does not use the disjointness of the open SUSY disks, i.e.,
the condition $z \ne 0$.
Thus, we obtain a well-defined map
\[
 m(0,\zeta)\colon \clF(\sD_r) \otimes \clF(\sD_r) \lto \clF(\sD_R)
\]
given by multiplication after extension.
By uniqueness of holomorphic continuation, 
the holomorphic map $m(Z)=m(z,\zeta)$ on $0<|z|<\ve$ extends 
holomorphically across $z=0$ with value $m(0,\zeta)$. 
Hence $m(Z)$ has no pole at $z=0$.

Let us now recall the state-superfield correspondence 
\eqref{eq:1:YaZ}, \eqref{eq:1:YaZb}: 
\[
 Y(a;Z)b \ceq \bigl[m(z,\zeta)(\wt{a} \otimes \wt{b})\bigr] \in \bbV
 \qquad (0<|z|<\ve).
\]
Since $m(Z)$ extends holomorphically to $z=0$ 
and depends polynomially on $\zeta$,
the cohomology class $Y(a;Z)b$ is holomorphic in $z$ at $0$ 
and polynomial in $\zeta$.
Therefore, we have the desired property \eqref{eq:1:Yser},
and the first half is over. 

\medskip

We now turn to the second half.
Let $T$ and $S$ be the even and odd translation operators induced by 
$\pd_z$ and $D=\pd_\zeta+\zeta\pd_z$, respectively.
It remains to show that $(\bbV,T,S,\cdot,[-_\Lambda-])$ satisfies:
(i) sesquilinearity,
(ii) SUSY skew-symmetry,
(iii) SUSY Jacobi identity, and
(iv) Leibniz rule.
We only sketch the proof below.
\begin{clist}
\item 
Sesquilinearity. 
We need to prove 
\[
 \{(Ta)_\Lambda b\} = -\lambda\{a_\Lambda b\}, \quad 
 \{(Sa)_\Lambda b\} = -\chi   \{a_\Lambda b\},
\]
and similarly for $T,S$ acting on the second entry 
(with $T+\lambda$, $S+\chi$). 
These are shown by using the already-established SUSY translation axioms 
\[
 [T,Y(a;Z)] = \pd_z Y(a;Z), \quad 
 [S,Y(a;Z)] = (\pd_\zeta+\zeta\pd_z)Y(a;Z),
\]
together with the fact that the Poisson bracket on $\clF(U)$ 
is compatible with the symmetry action 
(\cref{dfn:1:SPPFA} \ref{i:dfn:1:SPPFA:mor}), 
and the differentiation under the super-residue 
in the definition \eqref{eq:1:P}:
\[
 \sres_{Z=0} e^{z\lambda+\zeta\chi}\pd_z(\cdots) = 
-\lambda \sres_{Z=0} e^{z\lambda+\zeta\chi}(\cdots), \qquad
 \sres_{Z=0} e^{z\lambda+\zeta\chi}\pd_\zeta(\cdots) = 
-\chi \sres_{Z=0} e^{z\lambda+\zeta\chi}(\cdots).
\]

\item 
SUSY skew-symmetry: 
\begin{align}\label{eq:1:Ssks}
 \{a_\Lambda b\} = -(-1)^{p(a) p(b) + 1} \{b_{-\Lambda-\nabla}a\}, \quad 
 -\Lambda-\nabla\ceq(-T-\lambda,-S-\chi),
\end{align}
where $p(a),p(b)$ denote the parity of $a,b$.
We consider the two-point Poisson operation $\{\cdot,\cdot\}$ 
on the configuration space with two insertion disks around 
$(z_1,\zeta_1)$ and $(z_2,\zeta_2)$.
There are two expansions corresponding to the two regions 
$|z_1|>|z_2|$ and $|z_2|>|z_1|$, related by exchanging the two disks.
Factorization and commutativity provide the interchange isomorphism. 
Then, we use the graded skew-symmetry 
of the Poisson bracket at the cochain level $\clF(U)$:
\begin{align}\label{eq:1:sdsks}
 \{A,B\}=-(-1)^{\abs{A} \abs{B}}\{B,A\},
\end{align}
where $A,B$ denote cocycles representing $a,b$, and 
$\abs{A},\abs{B}$ denote the cohomological degrees.
Now, the definition \eqref{eq:1:P} implies that the exchange 
$Z \leftrightarrow W$ yields the desired equality \eqref{eq:1:Ssks}.
Note that the sign in \eqref{eq:1:Ssks} arises from 
the sign in \eqref{eq:1:sdsks}, using $(-1)^{\abs{A}}=(-1)^{p(a)}$ and 
$(-1)^{\abs{B}}=(-1)^{p(b)}$ (\cref{thm:1:533} \ref{i:thm:1:533:parity}),  
together with an additional factor $-1$ from 
taking the $\zeta$-coefficient in \eqref{eq:1:P}.

\item SUSY Jacobi identity: 
\begin{align}\label{eq:1:SJ}
 \{a_\Lambda \{b_\Gamma c\}\} = 
 (-1)^{p(a)+1}          \{\{a_\Lambda b\}_{\Gamma+\Lambda}c\} +
 (-1)^{(p(a)+1)(p(b)+1)}\{b_\Gamma\{a_\Lambda c\}\}.
\end{align}
We consider the three-point Poisson operation on the configuration space of 
three disjoint SUSY disks in a big disk. 
Factorization associativity gives equality of the two iterated products 
corresponding to different parenthesizations on overlaps of the domains.
This is the same analytic continuation mechanism used to prove 
the locality in \cref{thm:1:533}, Step 4, or \cite[Proposition 5.3.6]{CG1}.
Then, we use the Jacobi identity of the dg Poisson bracket on $\clF(U)$:
\begin{align}\label{eq:1:dJ}
 \{A,\{B,C\}\} = \{\{A,B\},C\} + (-1)^{\abs{A}\abs{B}}\{B,\{A,C\}\},
\end{align}
applied to the three inserted cochains 
supported in the three disks (extended into the big disk).
This gives the Jacobi identity for the three-point Poisson OPE 
at the cochain level. 
Now, pass to cohomology and take the super-residues defining $[a_\Lambda b]$, 
The standard residue manipulations turn the Jacobi identity 
on the cochain level into the desired identity \eqref{eq:1:SJ}.
As in the proof of skew-symmetry,
the sign difference between \eqref{eq:1:SJ} and \eqref{eq:1:dJ} 
comes from taking the super-residue.

\item 
Leibniz rule:
\begin{align}\label{eq:1:L}
 \{a_\Lambda (b\cdot c)\} = 
 \{a_\Lambda b\}\cdot c + (-1)^{|a||b|} b\cdot \{a_\Lambda c\}.
\end{align}
The commutative factorization algebra structure on $\clF$ 
gives a graded-commutative product on each $\clF(U)$.
By \cref{dfn:1:SPPFA} \ref{i:dfn:1:SPPFA:obj}, 
the dg Poisson bracket on $\clF(U)$ is 
assumed to be a derivation for that product:
\[
 \{A,BC\} = \{A,B\}C + (-1)^{\abs{A}\abs{B}}B\{A,C\}.
\]
Since all structure maps are graded Poisson algebra maps 
(\cref{dfn:1:SPPFA} \ref{i:dfn:1:SPPFA:mor}), 
this property is compatible with extending cochains 
and with factorization products.
Now, we apply the residue definition \eqref{eq:1:P} of $\{a_\Lambda-\}$ to 
the cocycle representing $b \cdot c$ 
(product in the commutative SUSY vertex algebra $\bbV$). 
The derivation property passes through
the residue operation to yield the identity \eqref{eq:1:L}.
\end{clist}
Thus, the proof is finished.
\end{proof}

\section{\texorpdfstring{$N_K=1$}{$N=1$} source and linear target}
\label{s:2}

We consider the free holomorphic sigma model 
whose source is the standard $N=1$ SUSY curve $\sC$
and the target is the linear space $\bbC^{n}$.
We formulate it in the BV framework 
using superfields $\Phi^i$ and their conjugates $\Pi_i$,
with action given by the free holomorphic $bc$--$\beta\gamma$ system.
The classical observables form an $N=1$ SUSY Poisson factorization algebra 
on $\sC$, which is holomorphically super-translation invariant
and $S^{1}$-equivariant. 
After BV quantization, the local cohomology acquires a natural $N_{K}=1$ 
SUSY vertex algebra structure, identified with the free-field realization 
underlying the chiral de Rham complex.

\subsection{The source as a derived space}\label{ss:2.1}

We follow \cite[Part 1, Chap.\ 2--3]{D+} 
for the terminology on super (differential) geometry.
Let $\Sigma = \bbC$ be the complex line with holomorphic coordinate $z$
and anti-holomorphic $\ol{z}$. 
Consider the dg ringed supermanifold $\sSig \ceq (\Sigma,A_{\sSig})$, where 
\[
 A_{\sSig} \ceq 
 \bigl(\Omega^{0,*}(\Sigma) \otimes \Wedge(\zeta),d_{\sSig}\bigr), 
 \quad 
 d_{\sSig} \ceq \pd_{\ol{z}}
\]
is the Dolbeault complex of smooth forms of Hodge type $(0,*)$ 
added with one odd coordinate $\zeta$.
The symbol $\Lambda(\zeta)$ denotes 
the polynomial superspace generated by $\zeta$,
so that we have $\Lambda(\zeta)=\bbC+\bbC\zeta$.
The space $A_{\sSig}$ has two gradings. 
One is the cohomological (form) degree valued in $\bbZ$: 
$\abs{dz} = \abs{d\ol{z}} = 1$ and $\abs{z}=\abs{\ol{z}}=\abs{\zeta}=0$,
extend additively.
The other is the parity (supe degree) valued in $\Zt=\{\ev,\od\}$: 
$p(z)=p(\ol{z})=\ev$ and $p(\zeta)=p(dz)=p(d\ol{z})=\od$, 
extend by the Koszul sign rule.

\subsection{Derived maps as cdga morphisms and superfields}

We consider a map $\Phi$ from $\sSig$ to the linear target $X=\bbC^n$ 
(regarded as an algebraic variety) in the derived context.
In other words, we consider a cdga morphism 
$\Phi^\sharp\colon (\shO(X),0) \to (A_{\sSig},d_{\sSig})$.
We denote the corresponding derived mapping stack by 
\[
 \clM \ceq \Map(\sSig, X).
\]
Since $\shO(X)=\bbC[x^1,\dotsc,x^n]$, 
choosing $\Phi^{\sharp} \in \clM$ is equivalent to choosing the superfields
\[
 \Phi^i \ceq \Phi^\sharp(x^i)\in A_{\sSig} \quad (i=1,\dotsc,n),
\]
subject to the compatibility 
$d_{\sSig}(\Phi^i)=d_{\sSig}(\Phi^\sharp(x^i))=\Phi^\sharp(0)=0$ for any $i$.
Using the physics notation, we can expand $\Phi^i$ as
\begin{align}\label{eq:1:Phii}
 \Phi^i(z,\ol{z},\zeta) = \gamma^i(z,\ol{z}) + \zeta c^i(z,\ol{z})
\end{align}
with $\gamma^i,c^i\in\Omega^{0,*}(\Sigma)$.
The parity of $\gamma^i$ is even, and $c^i$ is odd.
Let us bundle the components $\Phi^i$ into a single superfield, 
which, by abuse of notation, we also denote by
\begin{align}\label{eq:1:Phi}
 \Phi \in A_{\sSig} \otimes V, \quad 
 V \ceq \bbC^n \ (\text{regarded as a linear space}).
\end{align}


\subsection{Cotangent theory and conjugate superfields}\label{ss:1:BVth}

Following \cite{CG2}, we consider the cotangent theory to the above mapping space,
which is encoded by 
the shifted cotangent $\clF$ of the derived mapping stack $\clM$: 
\begin{align}\label{eq:1:clF}
 \clF \ceq T^*[-1]\clM = T^*[-1]\Map(\sSig,X).
\end{align}
A point of $\clF$ is a pair $(\Phi,\Pi)$ consisting of 
a map $\Phi$ as above and a degree-shifted cotangent vector $\Pi$ at $\Phi$.
For the linear target $X=\bbC^n$, we can model this as
\begin{align}\label{eq:1:B}
 \Pi \in A_{\sSig} \otimes V^\vee[1] \otimes \Ber(\sSig)
\end{align}
where $V^\vee$ is the linear dual of the linear space $V$ 
appearing in \eqref{eq:1:Phi},
and $\Ber(\sSig)$ is the Berezinian density on the supermanifold $\sSig$
(see \cref{ss:1:BVform} below).
Now, note that the components $\Phi^i$ ($i=1,\dotsc,n$) 
of $\Phi$ correspond to choosing a basis of $V$.
Then, the dual basis of $V^\vee$ gives components $\Pi_i$ of $\Pi$.
Using the physics notation, we can expand each component $\Pi_i$ as
\begin{align}\label{eq:1:Pii}
 \Pi_i(z,\ol{z},\zeta) = b_i(z,\ol{z}) + \zeta \beta_i(z,\ol{z})
\end{align}
with $b_i,\beta_i\in\Omega^{1,*}(\bbC)$.
The parity of $\beta_i$ is even, and $b_i$ is odd.

Another explanation of the components $\Pi_i$ is in order.
Consider the shifted cotangent $T^*[-1]X$ of the linear target $X=\bbC^n$. 
We have the standard coordinates $x^i$ of cohomological degree $0$ 
and the conjugate variables $p_i$ of cohomological degree $-1$.
Then, the superfield conjugate to $\Phi^i$ is given by 
$\Pi_i \ceq (\Phi^\sharp)^*(p_i)$, 
and we can bundle them as a superfield \eqref{eq:1:B}.

\subsection{BV symplectic form}\label{ss:1:BVform}

The cotangent theory $\clF=T^*[-1]\clM$ has a canonical BV pairing $\omega$. 
To explain that, let us consider the universal evaluation
$\evl\colon \sSig \times \clF \to T^*[-1]X$ 
covering $\sSig \times \clM \to X$.
Then, the canonical $(-1)$-shifted symplectic form $\omega_X$ on $T^*[-1]X$ 
induces the BV symplectic form $\omega$ on $\clF$ by 
\[
 \omega \ceq \int_{\sSig} \evl^*(\omega_X).
\]
Here, the integration along $\sSig$ means that 
we take the pushforward $\pi_*$ along the projection 
$\pi\colon \sSig \times \clF \to \clF$, determined by a choice of orientation.
A natural choice is the Berezinian density $\Ber(\sSig)$, 
and hereafter we use it.

Using the notation \eqref{eq:1:Phi} and \eqref{eq:1:B}, 
the BV symplectic form $\omega$ is expressed as 
\begin{align}\label{eq:1:omega}
 \omega = \int_{\sSig} \pair{ \delta B, \delta \Phi },
\end{align}
where $\pair{-,-}\colon V^\vee \otimes V \to \bbC$ is the evaluation pairing
and $\delta$ is the de Rham operator on superfield space, 
i.e., the variation of superfields. 

We can also express the form $\omega$ 
in terms of the components \eqref{eq:1:Phii} and \eqref{eq:1:Pii}.
The symplectic form on $T^*[-1]X$ can be written as 
$\omega_X=\delta p_i \, \delta x^i$,
where $x^i,p_i$ are as in \cref{ss:1:BVth},  
$\delta$ denotes the de Rham differential on $X$,
and Einstein's summation convention is used. 
In local coordinates, if we trivialize $\Ber(\sSig)$
by the Berezinian density $dz\,d\ol{z}\,d\zeta$, then
\[
 \omega = \int dz\,d\ol{z}\,d\zeta\;
 \delta \Pi_i(z,\ol{z},\zeta) \, \delta \Phi^i(z,\ol{z},\zeta),
\]
with the usual Koszul signs implied.
Using the expansions \eqref{eq:1:Phii}, \eqref{eq:1:Pii} 
and integrating out $\zeta$, we have 
\[
 \omega = \int dz\,d\ol{z}\;
 (\delta\beta_i\,\delta\gamma^i + \delta b_i\,\delta c^i).
\]

\subsection{BV action functional and BV differential}\label{ss:2.5}

Using the source differential $d_{\sSig}=\pd_{\ol{z}}$, 
consider the action functional
\begin{align}\label{eq:2:S}
 S \ceq \int_{\sSig} \pair{ \Pi, \pd_{\ol{z}} \Phi },
\end{align}
where $\int_{\sSig} $, $\pair{\cdot,\cdot}$ and $\delta$ 
are the same one as \cref{ss:1:BVform}.
The expansions \eqref{eq:1:Phii} and \eqref{eq:1:Pii} yield
\[
 S = \int dz\,d\ol{z}\;(\beta_i \pd_{\ol{z}}\gamma^i + b_i \pd_{\ol{z}}c^i),
\]
which is the standard action functional of the $bc$-$\beta\gamma$ system.

The BV differential $Q$ is the Hamiltonian vector field associated to $S$, 
\[
 Q \ceq \{S,-\}_\omega = \omega(\delta S,\delta-),
\]
where $\omega$ is the shifted BV symplectic form \eqref{eq:1:omega} 
on $\clF=T^*[-1]\clM$. 
Equivalently, $Q$ is the vector field of cohomological degree $1$ 
on the superfield space determined by 
$\iota_Q \omega = \delta S$.
By the standard functional calculation, we find that 
$Q$ acts on superfields $\Phi$ and $\Pi$ by
$Q\Phi = \pd_{\ol{z}}\Phi$ and $Q\Pi = \pd_{\ol{z}}\Pi$.
%
%
%
Hence, we have 
\begin{align}\label{eq:2:Q=pdoz}
 Q = \pd_{\ol{z}},
\end{align}
and taking the $Q$-cohomology is just considering the holomorphic part.
In the physics terminology, we are employing the holomorphic twist.

\subsection{Classical observables}\label{ss:2.6}

Now we turn to the discussion on the factorization algebra of 
classical observables in our BV cotangent theory \eqref{eq:1:clF} 
with BV form \eqref{ss:1:BVform} and action functional \eqref{eq:2:S}
using the formalism developed in \cref{s:1}.
Hereafter, using the notation in \cref{ss:1.1}, 
we take the source $\sSig$ to be
\[
 \sSig \ceq (\sC,D_Z), \quad D_Z = \pd_\zeta+\zeta\pd_z.
\]

Following the notation in \cite{CG1,CG2}, for an open $U \subset \sC$, 
we define the dg superspace (complex of linear superspaces) of superfields as 
\[
 \clE(U) \ceq 
 \bigl(\Omega^{0,*}(U_{\tred}) \otimes \Wedge(\zeta) \otimes V\bigr) \oplus
 \bigl(\Omega^{1,*}(U_{\tred}) \otimes \Wedge(\zeta) \otimes V^\vee\bigr),
 \quad 
 Q \ceq \pd_{\ol{z}}.
\]
The superfield $\Phi$ lives in the first summand, and $\Pi$ in the second.
$\clE(U)$ is an $\shO_{\an}(U)$-module, 
so that $\clE$ is a sheaf of $\shO_{\an}(\sC)$-modules 
on the complex supermanifold $\sC$.
Similarly, for test superfields, we consider 
\[
 \clE_c(U) \ceq 
 \bigl(\Omega_c^{0,*}(U_{\tred}) \otimes \Wedge(\zeta) \otimes V\bigr) \oplus
 \bigl(\Omega_c^{1,*}(U_{\tred}) \otimes \Wedge(\zeta) \otimes V^\vee\bigr),
 \quad 
 Q \ceq \pd_{\ol{z}},
\]
where the subscript $c$ means compact-supported sections.
$\clE_c$ is an $\shO_{\an}(\sC)$-module.

As local classical observables, 
let us consider the polynomial functionals on $\clE_c(U)$.
Then, the dg superspace of classical observables 
on an open $U \subset\sC$ is defined to be 
\begin{align}\label{eq:2:cObs}
 \cObs{}(U) \ceq \bigl( \Sym(\clE_c(U)^\vee[-1]), d_{\tcl}\bigr), \quad 
 d_{\tcl} \ceq \{S,-\}_\omega.
\end{align}
Here the symbol $\Sym$ denotes the completed\footnote{We suppress 
 the completion symbol $\wh{\ }$ for simplicity.} symmetric algebra. 
Since $\{S,-\}_\omega = Q = \pd_{\ol{z}}$ on $\clE_c(U)$, 
the differential $d_{\tcl}$ can be restated as: 
$d_{\tcl}(\ell)(e) \ceq (-1)^{\abs{\ell}+1}\ell(Qe)$ 
for a linear observable $\ell \in \clE_c(U)^\vee[-1]$, 
and extended by the graded Leibniz rule.

Similarly as the standard argument \cite{CG1,CG2} for the non-SUSY case, 
the correspondence $U \mto \cObs{}(U)$ produces 
the $N=1$ SUSY factorization algebra $\cObs{}$ on $\sC$ 
(in the sense of \cref{dfn:1:SFA})
valued in the symmetric monoidal dg super category $\Ch(\DVS)$
of cochain complexes in differentiable vector superspaces
(see \cref{rmk:1:DVS}).
Let us give an outline of the argument.
\begin{itemize}
\item
For an inclusion of opens $U \inj V$ in $\sC$, 
we have the extension-by-zero map $\clE_c(U) \inj \clE_c(V)$, 
and the restriction map on duals produces 
the extension map $\cObs{}(U) \to \cObs{}(V)$.

Then, for pairwise disjoint open subsets $U_1,\dotsc,U_k$ in $\sC$,  
define the factorization product for $U \ceq U_1 \sqcup \dotsb \sqcup U_k$ as
\begin{align}\label{eq:2:cfp}
 \mu^{\tcl}_U\colon 
 \cObs{}(U_1) \otimes \dotsb \otimes \cObs{}(U_k) \lto 
 \cObs{}(U)^{\otimes k} \lto \cObs{}(U),
\end{align}
where the first arrow is the tensor\footnote{The tensor product 
 in \eqref{eq:2:cfp} is the completed one, but we suppress 
 the completion symbol $\wh{\ }$ for simplicity.} of 
the extension maps $\cObs{}(U_i) \to \cObs{}(U)$
and the second is the multiplication on $\cObs{}(U)$ 
as the symmetric algebra \eqref{eq:2:cObs}.

This gives rise to a prefactorization algebra on $\sC$ valued in $\Ch(\DVS)$.

\item
Imposing Weiss descent, one gets a factorization algebra $\cObs{}$ on $\sC$.
By the construction \eqref{eq:2:cfp}, 
the factorization product on $\cObs{}$ is commutative 
in the sense of \cref{dfn:1:comSPFA}.
\end{itemize}
We call $\cObs{}$ the \emph{factorization algebra of classical observables}.

Furthermore, the BV symplectic form $\omega$ induces 
a $1$-shifted Poisson bracket
\[
 \{-,-\}\colon \cObs{}(U) \otimes \cObs{}(U) \lto \cObs(U)[1],
\]
and the factorization algebra $\cObs{}$ is an $N=1$ SUSY Poisson 
factorization algebra in the sense of \cref{dfn:1:SPPFA}.

Let us describe the Poisson structure explicitly.
For test fields 
$f\in\Omega_c^{1,*}(U) \otimes \Lambda(\zeta)\otimes V^\vee$ and 
$g\in\Omega_c^{0,*}(U) \otimes \Lambda(\zeta)\otimes V$, 
the linear functionals $\Phi,\Pi$ are given as
\[
 \Phi(f) = \int_U dz\,d\ol{z}\,d\zeta\;\pair{f,\Phi}, \quad 
 \Psi(f) = \int_U dz\,d\ol{z}\,d\zeta\;\pair{\Pi,g},
\]
and these are generators of $\cObs{}(U)=\Sym(\clE_c(U)^\vee[-1])$.
Then, the BV bracket is 
\begin{align}\label{eq:1:P0}
 \{\Phi(f),\Pi(g)\} = \int_U dz\,d\ol{z}\,d\zeta\;\pair{f,g}, \quad 
 \{\Phi(f),\Phi(f')\} =  0, \quad 
 \{\Pi(g),\Pi(g')\} =  0.
\end{align}

\subsection{Holomorphic super-translation invariance}\label{ss:2.7}

We check the holomorphic super-translation invariance and 
the $S^1$-action equivariance of our factorization algebra $\cObs{}$.

For each open $U \subset \sC$, define the odd derivative $D$ 
on the coefficient algebra $C^\infty(U) \otimes \Lambda(\zeta)$ by
\[
 D \ceq \pd_\zeta + \zeta \pd_z,
\]
so that we have $D(\zeta)=1$ and $D^2=\pd_z$.
It extends to $\Omega^{0,*}(U)\otimes\Lambda(\zeta)$ and 
$\Omega^{1,*}(U)\otimes\Lambda(\zeta)$ by acting on coefficients 
(i.e., commuting with $d\ol{z}$ and $dz$).
Since $Q=\pd_{\ol{z}}$ by \eqref{eq:2:Q=pdoz} and $\pd_z=D^2$, we have 
\[
 [Q, D] = 0, \quad [Q,\pd_z]=0.
\]
Hence $D$ is a cochain map of the complex $(\clE(U),Q)$,
and induces a derivation of the complex $\cObs{}(U)$ 
commuting with the differential $d_{\tcl}$.
It is compatible with the factorization product \eqref{eq:2:cfp} of $\cObs{}$.
Hence, the factorization algebra $\cObs{}$ is invariant 
under the odd derivation $D$. 
In physics terminology, 
$D$ is a supersymmetry operator on classical observables.
In particular, $\cObs{}$ is invariant under $D^2=\pd_z$, 
i.e., $\cObs{}$ is holomorphically translation invariant 
in the sense of \cite{CG1,CG2}.

Next, consider the $S^1$ action on $\sSig=\sC$ by rotation 
\[
 R_\varphi\colon (z,\zeta) \lmto (e^{i\varphi}z,e^{i\varphi/2}\zeta) \quad
 (\varphi\in\bbR/2\pi\bbZ).
\]
It sends open subsets $U\subset\bbC$ to $e^{i\varphi}U$, 
and acts on $(0,*)$- and $(1,*)$-forms by pullback.
In particular, it acts on the superfield $\Phi(z,\ol{z},\zeta)$ by 
$(R_\varphi \Phi)(z,\ol{z},\zeta)
=\Phi(e^{-\varphi}z,e^{i\varphi}\ol{z},e^{-i\varphi/2}\zeta)$,
and on the superfield $\Pi=\Pi_zdz$ by 
$(R_\varphi \Pi)(z,\ol{z},\zeta)=(R_\varphi \Pi_z)(e^{i\varphi}dz)$.
Since both pullback on forms and the $\zeta$-scaling are holomorphic in $z$, 
we have $[\pd_{\ol{z}},R_\varphi]=0$.
Then, the induced map 
$(R_\varphi)^*\colon \cObs{}(R_\varphi(U)) \to \cObs{}(U)$
commutes with the differential $d_{\tcl}$.
Also, rotations preserve disjointness and inclusions, 
so the factorization product \eqref{eq:2:cfp} is compatible with $R_\varphi$.
Therefore, $\cObs{}$ is $S^1$-equivariant.


\subsection{SUSY Poisson vertex algebra of classical observables}
\label{ss:2.8}

By the argument in the previous subsections, 
we see that the factorization algebra $\cObs{}$ of classical observables 
satisfies the assumptions in \cref{thm:1:SPVA}. 
Hence, we have:

\begin{thm}\label{thm:2:V(cObs)}
There is a structure of an $N_K=1$ SUSY Poisson vertex algebra on the cohomology
\[
 \bbV(\cObs{}) = H^{\bl}(\cObs{}(\sD_1)).
\]
\end{thm}

The SUSY Poisson structure of $\bbV(\cObs{})$ is encoded by \eqref{eq:1:P0}.
Let us rewrite it in terms of SUSY variables.
For $Z_1=(z_1,\zeta_1)$ and $Z_2=(z_2,\zeta_2)$, 
the super-delta distribution is
\begin{align}\label{eq:2:Sdelta}
 \delta(Z_1-Z_2) \ceq (\zeta_1-\zeta_2) \delta(z_1-z_2).
\end{align}
Then, in terms of the components $\Phi^i$ \eqref{eq:1:Phii} 
and $\Pi_i$ \eqref{eq:1:Pii}, we can rewrite \eqref{eq:1:P0} as 
\[
 \{\Phi^i(Z_1),\Pi_j (Z_2)\} = \delta^i_j \delta(Z_1-Z_2), \quad
 \{\Phi^i(Z_1),\Phi^j(Z_2)\} =  0, \quad 
 \{\Pi_i (Z_1),\Pi_j (Z_2)\} =  0.
\]
Using the $N_K=1$ super-diagonal 
\[
 z_{12} \ceq z_1-z_2-\zeta_1\zeta_2, 
\]
this is also equivalent to the SUSY OPE:
\[
 \Phi^i(Z_1)\Pi_j(Z_2) \sim \frac{\delta^i_j}{z_{12}} = 
 \frac{\delta^i_j}{z_1-z_2}\left(1+\frac{\zeta_1\zeta_2}{z_1-z_2}\right),
\]
with other OPEs trivial.
We can further rewrite it in terms of $N_K=1$ SUSY Lambda bracket:
\[
 \{\Phi^i_\Lambda \Pi_j\} = \delta^i_j.
\]
It is nothing but the SUSY Poisson structure in the $bc$-$\beta\gamma$ system 
(for the quantized structure, see \cite[Example 5.4]{HK}).

\subsection{BV quantization and propagator}
\label{ss:2.9}

We now apply the BV quantization formalism to 
the classical theory constructed in the previous subsections. 
Recall that, for an open subset $U \subset \sC$, the space of fields is
\[
 \clE(U) = \Omega^{0,*}(U_{\tred}) \otimes \Lambda(\zeta)\otimes V \;\oplus\;
           \Omega^{1,*}(U_{\tred}) \otimes \Lambda(\zeta)\otimes V^\vee,
\]
equipped with the BV differential
\[
 Q = \bpd \ceq \pd_{\ol{z}},
\]
and the BV pairing induced by the evaluation pairing 
$V^\vee\otimes V\to \bbC$ together with integration over $\sSig=\sC$. 
The action functional is
\[
 S(\Phi,\Pi) = \int_{\sC} \langle \Pi,\bpd \Phi\rangle.
\]
Since $S$ is quadratic in the fields, 
the theory is BV free in the sense of \cite[Definition 7.2.1]{CG2}. 
Therefore, its BV quantization is governed by 
a choice of gauge-fixing operator and the associated heat-kernel homotopy.

Following \cite[\S7.2]{CG2}, we equip $(\sC)_{\tred}=\bbC$ 
with the standard flat Hermitian metric and choose the gauge-fixing operator
\begin{align}\label{eq:2:QGF}
 Q^{\GF} \ceq \bpd^{\,*},
\end{align}
the formal adjoint of $\bpd$ on the reduced manifold $\bbC$ with metric. 
We extend $Q^{\GF}$ trivially to the odd variable $\zeta$ 
and coefficient spaces $V$, $V^\vee$. 
Then the generalized Laplacian, i.e., the commutator
\begin{align}\label{eq:2:H}
 H \ceq [Q,Q^{\GF}] = [\bpd,\bpd^{\,*}]
\end{align}
is the Dolbeault Laplacian acting coefficient-wise on $\clE(U)$.

For $t>0$, let $K_t$ denote the heat operator $e^{-tH}$, 
and write also $K_t(Z_1,Z_2)$ for its Schwartz kernel. 
In the language of BV quantization of free theories \cite[\S8.2]{CG2}, 
the regularized propagator is the operator
\begin{align}\label{eq:2:PeL}
 P_{\ep<L} \ceq \int_\ep^L Q^{\GF} e^{-tH}\,dt \quad (0<\ep<L<\infty).
\end{align}

\begin{lem} 
For each $0<\ep<L<\infty$, the operator $P_{\ep<L}$ is 
a well-defined smoothing operator of cohomological degree $-1$. 
Moreover, it satisfies the homotopy identity
\begin{align}\label{eq:2:hot}
 [Q,P_{\ep<L}] = e^{-\ep H}-e^{-LH}.
\end{align}
In particular, $P_{\ep<L}$ is a heat-kernel regularized homotopy inverse 
for the differential $Q=\bpd$.
\end{lem}

\begin{proof}
This lemma is more or less a standard one, 
but we give a proof for completeness.
By construction, both $Q=\bpd$ and $Q^{\GF}=\bpd^{\,*}$ 
act only on the Dolbeault part of the space of fields, and do not affect
either the odd variable $\zeta$ or the coefficients in $V,V^\vee$. 
Hence
\[
 H = [Q,Q^{\GF}] = \bpd\,\bpd^{\,*} + \bpd^{\,*}\bpd
\]
is the Dolbeault Laplacian acting coefficient-wise on $\clE(U)$.

For each $t>0$, the heat operator $e^{-tH}$ is smoothing, 
and therefore $Q^{\GF}e^{-tH}$ is again smoothing. 
Hence the integral
$P_{\ep<L} = \int_\ep^L Q^{\GF}e^{-tH}\,dt$ in \eqref{eq:2:PeL} 
converges in the topology of smoothing operators, and defines a continuous map
\[
 P_{\ep<L}\colon \clE_c(U) \lto \clE(U)[-1].
\]
Its cohomological degree is $-1$, 
since $Q^{\GF}$ has degree $-1$ and $e^{-tH}$ has degree $0$.

It remains to prove \eqref{eq:2:hot}. 
Since $Q$ commutes with $H=[Q,Q^{\GF}]$, 
it also commutes with the heat operator $e^{-tH}$. 
Therefore
$[Q,Q^{\GF}e^{-tH}] = [Q,Q^{\GF}]\,e^{-tH} = He^{-tH}$.
Integrating it from $\ep$ to $L$, we obtain
\[
 [Q,P_{\ep<L}] = \int_\ep^L He^{-tH}\,dt =
 -\int_\ep^L \frac{d}{dt}e^{-tH}\,dt =
 e^{-\ep H}-e^{-LH},
\]
where we used $He^{-tH}=-\frac{d}{dt}e^{-tH}$ in the second equality.
\end{proof}

The operator $P_{\ep<L}$ admits a Schwartz kernel, again denoted by
$P_{\ep<L}(Z_1,Z_2)$ for $Z_i=(z_j,\zeta_j)$, $j=1,2$.
By definition, this kernel is expressed as 
\begin{align*}
 P_{\ep<L}(Z_1,Z_2) = \int_\ep^L (Q^{\GF}_{Z_1}\otimes 1)\,K_t(Z_1,Z_2)\,dt.
\end{align*}
Then, the homotopy identity \eqref{eq:2:hot} is equivalently expressed 
at the level of kernels as
\begin{align}\label{eq:2:ker-hot}
 Q_{Z_1}P_{\ep<L}(Z_1,Z_2) = K_\ep(Z_1,Z_2)-K_L(Z_1,Z_2).
\end{align}
Thus $P_{\ep<L}$ is the homotopy inverse of $Q=\bpd$ regularized by the heat-kernel.

We next identify the singular part of the propagator on the flat superplane. 
Recall from \eqref{eq:1:Z12} the super-translation invariant difference
\[
 Z_{12} \ceq (z_{12},\zeta_{12})=(z_1-z_2-\zeta_1\zeta_2,\ \zeta_1-\zeta_2),
\]
and from \eqref{eq:2:Sdelta} the super-delta distribution
\[
 \delta(Z_1-Z_2) \ceq (\zeta_1-\zeta_2)\delta(z_1-z_2).
\]

\begin{prp}\label[prp]{prp:2:sCauchy}
Modulo terms smooth along the super-diagonal, 
the singular part of the propagator kernel on $\bbC^{1|1}$ is 
represented by the super-Cauchy kernel
\[
 P(Z_1,Z_2)_{\sing} \ceq 
  \frac{1}{2\pi i}\, \frac{d\ol{z}_1}{z_{12}}
  \otimes \id_{V\otimes V^\vee}.
\]
Equivalently, in the sense of distributions,
\[
 \bpd_{Z_1}P(Z_1,Z_2)_{\sing} =
 \delta(Z_1-Z_2) \otimes \id_{V\otimes V^\vee}.
\]
\end{prp}

\begin{proof}
By holomorphic super-translation invariance, 
the singular part of the propagator may depend only on 
the super-translation invariant difference 
$Z_{12}=(z_1-z_2-\zeta_1\zeta_2,\ \zeta_1-\zeta_2)$. 
On the reduced complex line $(\sC)_{\tred}=\bbC$, 
the standard fundamental solution of $\bpd$ is the Cauchy kernel
$\bpd_{z_1}\left(\frac{1}{2\pi i}\frac{1}{z_1-z_2}\right)=\delta(z_1-z_2)$.
In the present SUSY setting, 
the reduced difference $z_1-z_2$ is replaced by the invariant combination
$z_{12}=z_1-z_2-\zeta_1\zeta_2$.
Thus the unique singular kernel 
compatible with holomorphic super-translations is
$\frac{1}{2\pi i} \frac{d\ol{z}_1}{z_{12}}$, 
and its $\bpd$-derivative is the super-delta distribution, namely
\[
 \bpd_{Z_1}\left(\frac{1}{2\pi i}\frac{d\ol{z}_1}{z_{12}}\right) =
 \delta(Z_1-Z_2).
\]
Finally, since the coefficient spaces contribute 
only through the evaluation pairing $V^\vee\otimes V\to \bbC$, 
the full kernel is obtained by tensoring with $\id_{V \otimes V^\vee}$. 
This proves the proposition.
\end{proof}

The proposition shows that the singular part of the propagator is 
precisely the kernel expected from the free $bc$-$\beta\gamma$ system. 
In the next \cref{ss:2.10},
we will formalize this in terms of 
the factorization algebra of quantum observables.

\subsection{Quantum observables and the associated SUSY vertex algebra}
\label{ss:2.10}

We now define the factorization algebra of quantum observables associated to 
the free BV theory of \cref{ss:2.9}, and then apply \cref{thm:1:533}
to obtain the corresponding $N_K=1$ SUSY vertex algebra.

Fix a scale $L>0$. 
Let $K_L$ be the heat kernel for the operator $e^{-LH}$, where
$H = [Q,Q^{\GF} ] =[\bpd,\bpd^{\,*}]$
is the Dolbeault Laplacian from \eqref{eq:2:H}. 
The smooth kernel $K_L$ determines the scale-$L$ BV Laplacian \cite[\S8.2]{CG2}: 
\[
 \Delta_L\colon \Sym\bigl(\clE_c(U)^\vee[1]\bigr) \lto
 \Sym\bigl(\clE_c(U)^\vee[1]\bigr)[1],
\]
which is the unique second-order differential operator vanishing on constants
and characterized on linear observables by the pairings induced by $K_L$.

Since the theory is free, there is no interaction term and 
the quantum master equation is automatic. 
Thus the scale-$L$ quantum observables on an open subset 
$U \subset \bbC^{1|1}$ are given by the BD algebra \cite[\S8.2]{CG2}: 
\begin{align}\label{eq:2:qObsL}
 \qObs{L}(U) \ceq 
 \Bigl(\Sym\bigl(\clE_c(U)^\vee[1]\bigr)\dbr{\hbar}, \, d_{\q,L}\Bigr),
 \qquad
 d_{\q,L} \ceq d_{\tcl}+\hbar\,\Delta_L,
\end{align}
where\footnote{As in \eqref{eq:2:cObs}, $\Sym$ denotes the completed 
symmetric algebra, and we suppress the completion symbol $\wh{\ }$.}
$d_{\tcl}$ is the classical BV differential on $\cObs{}(U)$ 
induced by $Q=\bpd$, given in \eqref{eq:2:cObs}.

For $0<\ep<L$, let $P_{\ep<L}$ be 
the regularized propagator from \eqref{eq:2:PeL}. 
If $U_1,\dotsc,U_k$ are pairwise disjoint opens in $\sC$, 
then the kernel $P_{\ep<L}$ is smooth on $U_i \times U_j$ for $i \ne j$. 
Hence it defines a contraction operator\footnote{As in \eqref{eq:2:cfp}, 
 the tensor product is the completed one, but we suppress $\wh{\ }$.}
\[
 \pd_{P_{\ep<L}}\colon 
 \qObs{L}(U_1) \otimes\dotsb\otimes \qObs{L}(U_k) \lto
 \qObs{L}(U_1) \otimes\dotsb\otimes \qObs{L}(U_k)[-1]
\]
by summing all pairwise contractions via $P_{\ep<L}$ 
between different tensor factors. 
Then, denoting $U \ceq U_1 \sqcup \cdots \sqcup U_k$,
we define the factorization product $\mu^{\q}_U$ of $\qObs{L}(U)$ by
\begin{align}\label{eq:2:qfm}
 \mu^{\q}_U \ceq \mu^{\tcl}_U \circ
 \exp\bigl(\hbar\,\pd_{P_{\ep<L}}\bigr)\colon 
 \qObs{L}(U_1) \otimes \dotsb \otimes \qObs{L}(U_k) \lto \qObs{L}(U),
\end{align}
where $\mu^{\tcl}$ is the factorization product on $\cObs{}$. 

Because the opens are pairwise disjoint, 
the kernel $P_{\ep<L}$ is smooth on the relevant product, 
so the right-hand side of \eqref{eq:2:qfm} is well-defined 
and independent of $\ep$ as $\ep\to 0$. 
Since the theory is free, no counterterms occur, and the family of products
\eqref{eq:2:qfm} satisfies the prefactorization axioms. 
Imposing Weiss descent, 
we obtain an $N=1$ SUSY factorization algebra on $\bbC^{1|1}$
over the coefficient ring $\bbC\dbr{\hbar}$, which we denote by
\[
 \qObs{}.
\]

\begin{prp}\label{prp:2:qObsFA}
The $N=1$ SUSY factorization algebra $\qObs{}$
satisfies the conditions \ref{i:1:u}--\ref{i:1:hol} of \cref{thm:1:533}.
Moreover, we have 
\begin{align}\label{eq:2:qObs/h=cObs}
 \qObs{}/\hbar\qObs{} \cong \cObs{}.
\end{align}
\end{prp}

\begin{proof}
Fix a scale $L>0$. 
By definition \eqref{eq:2:qObsL},
\[
 \qObs{L}(U) = 
 \Bigl(\Sym\bigl(\clE_c(U)^\vee[1]\bigr)\dbr{\hbar},\, d_{\q,L}\Bigr), \quad
 d_{\q,L} \ceq d_{\tcl}+\hbar\Delta_L,
\]
where $d_{\tcl}$ is the classical BV differential and 
$\Delta_L$ is the scale-$L$ BV Laplacian determined by the heat kernel $K_L$. 
Since the theory is free, there is no interaction functional, 
hence no counterterm or renormalization ambiguity. 
Therefore the quantum observables are obtained from the classical ones 
by adjoining the second-order BV operator $\hbar\Delta_L$, 
and their classical limit satisfies \eqref{eq:2:qObs/h=cObs}.

We now verify the conditions \ref{i:1:u}--\ref{i:1:hol} of \cref{thm:1:533}.

\smallskip

Condition \ref{i:1:u}: unitality.
For each open $U \subset \bbC^{1|1}$, the algebra $\qObs{}(U)$ is 
a unital BD algebra whose unit is the constant functional
\[
 1 \in \Sym^0\bigl(\clE_c(U)^\vee[1]\bigr)\dbr{\hbar} \cong \bbC\dbr{\hbar}.
\]
Since both $d_{\tcl}$ and $\Delta_L$ vanish on constants, 
we have $d_{\q,L}(1)=0$. 
Moreover, the factorization products preserve the unit: 
the classical product $\mu^{\tcl}$ is unital, 
and the contraction operator $\pd_{P_{\ep<L}}$ 
annihilates constants. Hence $\qObs{}$ is unital.

\smallskip

Condition \ref{i:1:SETI}: 
$S^1$-equivariance and holomorphic super-translation invariance.
Let $D = \pd_\zeta+\zeta\pd_z$ be the odd super-translation operator. 
It acts coefficient-wise on the field complex $\clE(U)$ 
and commutes with $Q=\bpd$. 
Since the BV pairing is defined by integration against 
the translation-invariant Berezinian density, 
the induced derivation on $\cObs{}(U)$ commutes with $d_{\tcl}$. 
The gauge-fixing operator $Q^{\GF}=\bpd^{\,*}$ is defined from
the flat translation-invariant metric on the $\bbC$, 
hence it commutes with the holomorphic translations and 
therefore so do the Laplacian $H=[Q,Q^{\GF}]$, the heat kernel $K_L$, 
the BV Laplacian $\Delta_L$, and the propagator $P_{\ep<L}$. 
Thus the induced derivations descend to $\qObs{}$, 
and $\qObs{}$ is holomorphically super-translation invariant.

Similarly, the rotation
$R_\phi\colon (z,\zeta) \mto (e^{i\phi}z,e^{i\phi/2}\zeta)$
preserves the flat metric, the Berezinian density, the operator $Q=\bpd$,
the gauge-fixing operator $Q^{\GF}$, the Laplacian $H$, 
the heat kernel, and the propagator. 
Therefore $\qObs{}$ is $S^1$-equivariant.

\smallskip

Condition \ref{i:1:di}: disk independence.
Let $\sD_r \inj \sD_s$ ($0 < r \le s$) be an inclusion of SUSY disks. 
On the reduced disk $D_t \subset \bbC$, 
the Dolbeault complex is acyclic in positive degree, 
and its degree-zero cohomology consists of holomorphic functions. 
Hence the inclusion induces a quasi-isomorphism on the field complexes, 
and therefore on the classical observables $\cObs{}(\sD_r)\to \cObs{}(\sD_s)$. 
The quantum differential differs from the classical one 
by the smoothing operator $\hbar\Delta_L$, 
which does not alter the local quasi-isomorphism type. 
Hence $\qObs{}(\sD_r) \lto \qObs{}(\sD_s)$ is a quasi-isomorphism. 

\smallskip

Condition \ref{i:1:rho(E)}: weight boundedness.
Let $E=z\pd_z+\hf\zeta\pd_\zeta$ be the infinitesimal generator 
of the $S^1$-rotation. 
By the $S^1$-equivariance established above, 
$E$ acts on $\qObs{}(\sD_1)$ and commutes with the differential. 
On local holomorphic jets at the origin, the coordinate $z$ has weight $1$, 
the odd coordinate $\zeta$ has weight $1/2$, 
and each holomorphic derivative $\pd_z$ raises weight by $1$. 
Hence, the generators of the free theory carry weights 
in a lower-bounded subset of $\hf\bbZ$, 
and each fixed weight space is finite-dimensional. 
Passing to cohomology, we obtain a locally finite, 
lower-bounded $\hf\bbZ$-grading on
$\bbV(\qObs{}) = H^\bullet\!\bigl(\qObs{}(\sD_1)\bigr)$,
so condition \ref{i:1:rho(E)} holds.

\smallskip

Condition \ref{i:1:hol}: holomorphic dependence.
Fix $0<r\ll R$ and a configuration $\ul{x}=(Z_1,\dotsc,Z_n)$, 
$Z_i=(z_i,\zeta_i)$, such that the translated SUSY disks $\sD_r(Z_i)$ are 
pairwise disjoint and contained in $\sD_R$. 
The corresponding $n$-point factorization map is obtained by 
translating observables from $\sD_r$ to each $\sD_r(Z_i)$, 
then applying Wick contractions determined by the propagator, 
and finally extending into $\sD_R$. 
Since the translations depend smoothly on $(z_i,\ol{z}_i)$ 
and polynomially on $\zeta_i$, 
and since the heat kernel and propagator are smooth away from the diagonal
and polynomial in the odd variables, 
the resulting multilinear map\footnote{As before, 
 the tensor product $\otimes n$ is the completed one, but we suppress $\wh{\ }$.}
\[
 m^q_{r,R,n}(\ul{x})\colon \qObs{}(\sD_r)^{\otimes n} \lto \qObs{}(\sD_R)
\]
depends smoothly on $(z_i,\ol{z}_i)$ and polynomially on $\zeta_i$.

It remains to prove the Cauchy--Riemann condition up to homotopy. 
The only anti-holomorphic dependence of the factorization product is 
through the smooth kernels appearing in the contractions. 
Differentiating with respect to $\ol{z}_i$ 
may therefore be moved onto the propagator. 
By the kernel form \eqref{eq:2:ker-hot} of the homotopy identity
\[
 Q_{Z_i}P_{\ep<L}(Z_i,Z_j) = K_{\ep}(Z_i,Z_j)-K_L(Z_i,Z_j),
\]
the anti-holomorphic derivative is $Q$-exact 
up to the difference of smooth heat kernels. 
Integrating these smooth kernels produces a multilinear operator
\[
 h^{(i)}_{r,R,n}(\ul{x})\colon 
 \qObs{}(\sD_r)^{\otimes n} \lto \qObs{}(\sD_R)[1]
\]
depending smoothly on $\ul{x}$, such that
\[
 \pd_{\ol z_i}m^q_{r,R,n}(\ul{x}) = 
 d_{\q}\,h^{(i)}_{r,R,n}(\ul{x}) + h^{(i)}_{r,R,n}(\ul{x})\,d_{\q}.
\]
Hence the induced operations on cohomology are holomorphic 
in the variables $z_i$, and condition {i:1:hol} is satisfied.

All the hypotheses of \cref{thm:1:533} are therefore verified. 
This completes the proof.
\end{proof}

Then, applying \cref{thm:1:533} to the factorization algebra $\qObs{}$, 
we have:

\begin{thm}\label{thm:2:qObsVA}
The cohomology 
\[
 \bbV(\qObs{}) \ceq H^\bullet\!\bigl(\qObs{}(\sD_1)\bigr)
\]
carries a natural structure of 
an $N_K=1$ SUSY vertex algebra over $\bbC\dbr{\hbar}$.
More precisely:
\begin{enumerate}
\item 
The parity on $\bbV(\qObs{})$ 
is induced by the cohomological grading modulo $2$.

\item 
The vacuum vector $\vac\in \bbV(\qObs{})$
is induced by the unit $1\in \qObs{}(\emptyset)\cong \bbC\dbr{\hbar}$.

\item 
The even and odd translation operators on $\bbV(\qObs{})$ are induced by
\[
 T \ceq \rho(\pd_z), \quad
 S \ceq \rho(D) = \rho(\pd_\zeta+\zeta\pd_z),
\]
and satisfy $S^2=T$ on cohomology.

\item 
For every $a\in \bbV(\qObs{})$, there is a corresponding superfield
\[
 Y(a;Z)=Y(a;z,\zeta) \in \End(\bbV(\qObs{}))\dbr{z^{\pm1}}[\zeta]
\]
obtained from the factorization products of $\qObs{}$.
\end{enumerate}
\end{thm}

\begin{proof}
The existence of a $N_K=1$ SUSY vertex algebra structure on $\bbV(\qObs{})$
is a corollary of \cref{prp:2:qObsFA} and \cref{thm:1:533}.
The descriptions of the parity, vacuum, and translation operators are 
exactly those given by \cref{thm:1:533}. 
The state-superfield correspondence is obtained from 
the two-disk factorization product. 
This proves the theorem.
\end{proof}

Next, we identify the resulting SUSY vertex algebra explicitly. 

\begin{cor}\label[cor]{cor:2:bcbg}
The $N_K=1$ SUSY vertex algebra $\bbV(\qObs{})$ is 
the $N_K=1$ SUSY vertex algebra of the free $bc$-$\beta\gamma$ system. 
More precisely, denoting by 
\[
 C^i(z,\zeta) = c^i(z)+\zeta\gamma^i(z), \quad 
 B_i(z,\zeta) = b_i(z)+\zeta \beta_i(z)  \quad (i=1,\dotsc,n)
\]
the classes in $\bbV(\qObs{})$ of linear quantum observables associated to 
the even superfield $\Phi^i$ \eqref{eq:1:Phii} 
and the odd $\Pi_i$ \eqref{eq:1:Pii},
their SUSY OPE are 
\begin{align}\label{eq:2:qOPE}
 C^i(Z_1) B_j(Z_2) \sim \hbar \frac{\delta^i_j}{z_{12}}, \quad
 C^i(Z_1) C^j(Z_2) \sim 0, \quad
 B_i(Z_1) B_j(Z_2) \sim 0
\end{align}
with $Z_k=(z_k,\zeta_k)$, $k=1,2$, and $z_{12} \ceq z_1-z_2-\zeta_1\zeta_2$.
Equivalently, the corresponding $N_K=1$ SUSY $\Lambda$-brackets are
\begin{equation}\label{eq:2:qLam}
 [C^i{}_\Lambda B_j] = \hbar \delta^i_j, \quad
 [C^i{}_\Lambda C^j] = 0, \quad
 [B_i{}_\Lambda B_j] = 0.
\end{equation}
\end{cor}

\begin{proof}
\cref{thm:2:qObsVA} gives $\bbV(\qObs{})$ 
the structure of an $N_K=1$ SUSY vertex algebra. 
By construction, the singular part of the operator product is 
determined by the singular part of the factorization product, 
which in the free theory is given by Wick contraction with the propagator. 
By \cref{prp:2:sCauchy}, the singular part of that propagator is
\[
 \frac{1}{2\pi i}\frac{d\ol{z}_1}{z_{12}} \otimes \id_{V \otimes V^\vee}.
\]
Hence, the only nontrivial contraction is between $\Phi_i$ and $\Pi_j$, 
and it contributes the singular term
\[
 \hbar\,\frac{\delta_{ij}}{z_{12}}.
\]
This proves \eqref{eq:2:qOPE}. 
The equivalent $\Lambda$-bracket description \eqref{eq:2:qLam} 
follows from the correspondence  \cite[\S3.2.1, \S4.2]{HK} between 
the singular part of the SUSY OPE and the $N_K=1$ SUSY $\Lambda$-bracket.
\end{proof}

\begin{rmk}
The $\Lambda$-brackets in \eqref{eq:2:qLam} is equivalent to 
the $bc$-$\beta\gamma$ system 
$(a^i_{\MSV},b^i_{\MSV},\phi^i_{\MSV},\psi^i_{\MSV})_{i=1}^N$ in \cite{MSV}, 
whose $N_K=1$ SUSY vertex algebra structure is given by 
$B^i_{\BHS}(Z) \ceq b^i_{\MSV}(z)+\zeta \phi^i_{\MSV}(z)$ and 
$\Psi^i_{\BHS}(Z) \ceq \psi^i_{\MSV}(z)+\zeta a^i_{\MSV}(z)$ 
in \cite[Example 3.12]{BHS}.
The equivalence is given by 
$C^i(Z) = B^i_{\BHS}(Z)$ and $B_i(Z) = \Psi^i_{\BHS}(Z)$.
In particular, the super-current
\[
 G(Z) \ceq \nod{B_i\,(S C^i)}(Z),
\]
where $\nod{\ }$ denotes the normally ordered product 
\cite[Definition 3.2.11]{HK},
gives an $N_K=1$ SUSY vertex algebra structure on $\bbV(\qObs{})$.
It can be expanded as
\[ 
 G(Z) = G(z) + \zeta T(z), \quad 
 G(z) \ceq \nod{c^i\,\beta_i}(z), \quad 
 T(z) \ceq \nod{\beta_i\,(\pd \gamma^i)}(z) + \nod{b_i\,(\pd c^i)}(z).
\]
\end{rmk}

\section{\texorpdfstring{$N_K=1$}{NK=1} source and non-linear target}
\label{s:3}

In this section, we extend the linear-target construction of \cref{s:2} 
to the case where the target is an arbitrary complex manifold $X$. 
The goal is to lift the $N=1$ superfield argument of 
Ben-Zvi--Heluani--Szczesny \cite{BHS} 
from the language of SUSY vertex algebras 
to the language of SUSY factorization algebras. 
Concretely, we will construct, on each holomorphic coordinate chart of $X$,
a local quantum $N_K=1$ SUSY factorization algebra modeled on 
the free $bc$--$\beta\gamma$ system of \cref{s:2}, 
and then study its descent under holomorphic changes of target coordinates.
The resulting local-to-global object recovers 
the chiral de Rham complex $\CdR_X$ in the form of a sheaf of SUSY vertex algebras.

Throughout this section, the source is $\sSig=(\sC,\clD_{\std})$ 
with $\clD_{\std}$ generated by $D_Z \ceq \pd_{\zeta}+\zeta\pd_{z}$,
and we continue to denote by $z_{12} \ceq z_1-z_2-\zeta_1\zeta_2$
the $N_K=1$ super-distance \eqref{eq:1:Z12}.

\subsection{Globalization to a complex manifold}\label{ss:3.1}

Let $X$ be a complex manifold of complex dimension $n$. 
Our aim is to describe the holomorphic sigma model with source $\sSig$ 
and target $X$ in the BV formalism, at least locally on $X$, 
and then to glue the resulting local factorization algebras.

Fix a holomorphic coordinate chart $(V;x^1,\dotsc,x^n)$, $V \subset X$.
Over $V$, the tangent bundle $\rst{T_X}{V}$ is trivialized by 
the coordinate frame $\{ \pd/\pd x^i \}_{i=1}^n$, 
and therefore the local theory is modeled on the linear target $\bbC^n$,
coincides with the $bc$--$\beta\gamma$ system in \cref{s:2}. 
In particular, on a SUSY disk $U \subset \sSig$, we consider even superfields
$\Phi^i(z,\ol{z},\zeta) = \gamma^i(z,\ol{z})+\zeta c^i(z,\ol{z})$
and odd conjugate superfields
$\Pi_i(z,\ol{z},\zeta) = b_i(z,\ol{z})+\zeta\beta_i(z,\ol{z})$ for $i=1,\dotsc,n$
%
The local classical action on $U$ is
\begin{align}\label{eq:3:laf}
 S_V(U) \ceq \int_U dz\,d\ol{z}\,d\zeta \; \Pi_i\,\bpd \Phi^i, \quad 
 \bpd \ceq \pd_{\ol{z}}, 
\end{align}
and the corresponding BV pairing is
\begin{align}\label{eq:3:omega}
 \omega_{V}(U) = \int_U dz\,d\ol{z}\,d\zeta \; \delta \Pi_i\,\delta\Phi^i.
\end{align}
Thus, in a coordinate chart, we are exactly in the situation of \cref{s:2}, 
with $V=\bbC^n$ and $V^{\vee}$ paired by evaluation.

The point of the $N=1$ superfield formalism is that 
the local fields glue more naturally across coordinate changes 
than the individual component fields. 
Let $(V_\alpha;x^i_\alpha)$ and $(V_\beta;x^i_\beta)$ 
be two holomorphic charts, and write the coordinate change 
on the overlap $V_{\alpha\beta} \ceq V_\alpha \cap V_\beta$ as
\[
 x^i_\beta=f^i_{\beta\alpha}(x_\alpha).
\]
Then the superfields transform by substitution 
into the holomorphic coordinate change:
\begin{align}\label{eq:3:sfc}
 \Phi^i_\beta = f^i_{\beta\alpha}(\Phi_\alpha), \quad
 \Pi_{\beta,i} = \frac{\pd x^j_\alpha}{\pd x^i_\beta}(\Phi_\alpha) \Pi_{\alpha,j}.
\end{align}
This is the fundamental simplification of the $N=1$ formalism: 
the curved target geometry is absorbed into a tensorial transformation law 
for the superfields, exactly as in \cite[\S4]{BHS}. 
In particular, the coordinate changes preserve 
the local OPE singularities of the free theory.

We will use \eqref{eq:3:sfc} to 
descend the local factorization algebras from charts to $X$.

\subsection{Local non-linear structure and coordinate changes}\label{ss:3.2}

We now explain why the transformation law \eqref{eq:3:sfc} 
is compatible with the local BV theory.

Let $V \subset X$ be a holomorphic chart. 
The sheaf of superfields on a SUSY disk $U \subset \sSig$ is
\begin{align*}
 \clE_V(U) \ceq
 \Omega^{0,*}(U_{\tred}) \otimes \Lambda(\zeta) \otimes  \bbC^n \;\oplus\;
 \Omega^{1,*}(U_{\tred}) \otimes \Lambda(\zeta) \otimes (\bbC^n)^\vee
\end{align*}
with differential $Q=\bpd$. 
The first summand is spanned by the $\Phi^i$, 
and the second by the $\Pi_i$. 
The BV pairing \eqref{eq:3:omega} identifies 
the two summands as shifted duals, exactly as in \cref{s:2}.

Suppose now that $x_\beta=f_{\beta\alpha}(x_\alpha)$ is a holomorphic change of
target coordinates on an overlap $V_{\alpha\beta} \ceq V_\alpha \cap V_\beta$.
The assignment \eqref{eq:3:sfc} induces an isomorphism of sheaves of dg superspaces
\begin{align}\label{eq:3:fiso}
 \varphi^{\tcl}_{\beta\alpha}\colon 
 \rst{\clE_{V_\alpha}}{V_{\alpha\beta}} \lsto
 \rst{\clE_{V_\beta }}{V_{\alpha\beta}}.
\end{align}
Since $f_{\beta\alpha}$ is holomorphic, we have
$\bpd \Phi^i_\beta = 
 \frac{\pd f^i_{\beta\alpha}}{\pd x^{j}_\alpha}(\Phi_\alpha)\,
 \bpd \Phi^{j}_\alpha$.
Hence
\[
 \Pi_{\beta,i}\,\bpd \Phi^i_\beta =
 \frac{\pd x^j_\alpha}{\pd x^i_\beta }(\Phi_\alpha) \, \Pi_{\alpha,j} \cdot
 \frac{\pd x^i_\beta }{\pd x^k_\alpha}(\Phi_\alpha) \, \bpd \Phi^k_\alpha =
 \Pi_{\alpha,j} \, \bpd \Phi^j_\alpha,
\]
so the classical action $S_V$ \eqref{eq:3:laf} is invariant 
under target coordinate changes. Similarly,
\[
 \delta \Pi_{\beta, i} \, \delta \Phi^i_\beta  =
 \delta \Pi_{\alpha,j} \, \delta \Phi^j_\alpha, 
\]
and the BV pairing $\omega_V$ \eqref{eq:3:omega} is preserved.

Therefore, the local classical theory on coordinate charts 
glues strictly at the classical level, and we have:

\begin{prp}\label{prp:3:cgl}
The assignments $(V \subset X) \mto (\clE_V,Q,\omega_V,S_V)$
form a holomorphic atlas of local BV theories on $X$. 
On overlaps, the transition maps are given by \eqref{eq:3:fiso}, 
and these preserve both the BV pairing and the classical action.
\end{prp}


In components, the transformation law \eqref{eq:3:sfc} reproduces 
the non-linear corrections of the curved $bc$--$\beta\gamma$ system
(see \cite[\S3.5]{MSV} for example). 
Indeed, writing $\Phi^i=\gamma^i+\zeta c^i$ 
and $\Pi_i=b_i+\zeta \beta_i$, we obtain
\begin{align*}
 \gamma^i_\beta &=
 f^i_{\beta\alpha}(\gamma_\alpha), \\
 c^i_\beta &=
 \frac{\pd f^i_{\beta\alpha}}{\pd x^j_\alpha}(\gamma_\alpha)\,c^j_\alpha, \\
 b_{\beta,i} &=
 \frac{\pd x^j_\alpha}{\pd x^i_\beta}(\gamma_\alpha)\,b_{\alpha,j}, \\
 \beta_{\beta,i} &=
 \frac{\pd x^j_\alpha}{\pd x^i_\beta}(\gamma_\alpha)\,\beta_{\alpha,j} +
 \frac{\pd^2x^j_\alpha}{\pd x^i_\beta\pd x^{k}_\beta}(\gamma_\alpha)
 \frac{\pd x^{k}_\beta}{\pd x^{\ell}_\alpha}(\gamma_\alpha)\,
 c^{\ell}_\alpha b_{\alpha,j}.
\end{align*}

\subsection{Local classical observables and SUSY factorization algebras}
\label{ss:3.3}

We now apply the formalism of \cref{s:1} and \cref{s:2} chart-wise on $X$.

Fix a holomorphic chart $V \subset X$. 
For each open $U \subset \sSig$, define the classical observables on $U$ by
\begin{align*}
 \cObs{V}(U) \ceq \Sym\bigl(\clE_{V,c}(U)^{\vee}[1]\bigr), \quad
 d^{\tcl} \ceq \{S_V,-\}_{\omega_V},
\end{align*}
where $\clE_{V,c}(U)$ denotes compactly supported fields on $U$, 
and $\Sym$ is the completed symmetric algebra. 
Exactly as in \cref{ss:2.5} \eqref{eq:2:Q=pdoz}, 
the differential is induced by the Dolbeault operator:
\[
 d^{\tcl} = \bpd.
\]

The local observables form 
an $N=1$ SUSY Poisson factorization algebra on $\sSig$:
\begin{align}\label{eq:3:lsFA}
 U \lmto \cObs{V}(U).
\end{align}
The factorization product is defined by extension by zero and multiplication
in the symmetric algebra, just as in \eqref{eq:2:cfp} of \cref{ss:2.6}. 
The Poisson bracket is induced by the BV pairing \eqref{eq:3:omega}. 
In terms of linear observables $C^i,B_i$ defined as in \cref{cor:2:bcbg}, 
we have the classical brackets
\begin{align*}
 \{C^i(Z_1),B_j(Z_2)\} = \delta^i_j\,\delta(Z_1-Z_2), \quad
 \{C^i(Z_1),C^j(Z_2)\} = 0, \quad
 \{B_i(Z_1),B_j(Z_2)\} = 0.
\end{align*}

\begin{prp}\label[prp]{prp:3:scObs}
For every holomorphic coordinate chart $V\subset X$, 
the assignment \eqref{eq:3:lsFA} defines 
an $N=1$ SUSY Poisson factorization algebra on $\sSig$. 
Moreover, the transition maps \eqref{eq:3:fiso} 
induce isomorphisms of Poisson factorization algebras on overlaps:
\[
 \varphi^{\tcl}_{\beta\alpha}:
 \Obs^{\tcl}_{V_\alpha}|_{V_{\alpha\beta}}
 \stackrel{\sim}{\longrightarrow}
 \Obs^{\tcl}_{V_\beta}|_{V_{\alpha\beta}}.
\]
Hence the classical observables glue to 
a sheaf of $N=1$ SUSY Poisson factorization algebras on $X$,
which is denoted by 
\[
 \scObs{X}.
\]
\end{prp}

\begin{proof}
On each chart this is exactly the construction of 
\S\S \ref{ss:2.1}--\ref{ss:2.8}, with the vector space target $\bbC^n$ 
replacing the tangent space in coordinates.
The compatibility with coordinate changes follows from \cref{prp:3:cgl},
since the transition maps preserve the BV action and the BV pairing, 
hence also the induced differential and Poisson bracket.
\end{proof}

Applying \cref{thm:1:SPVA} chart-wise, we obtain a sheaf $\clV^{\tcl}_X$
of local $N_K=1$ SUSY Poisson vertex algebras on $X$. 
Concretely, for each chart $(V;x^i_\alpha)$, letting
\[
 \clV^{\tcl}_X(V) \ceq \bbV\bigl(\cObs{V}\bigr) 
                  = H^{\bl}\bigl(\cObs{V}(\sD_1)\bigr),
\]
we obtain a commutative $N_K=1$ SUSY vertex algebra 
with generators $C^i,B_i$ and $\Lambda$-brackets
\begin{align}\label{eq:3.3-classical-lambda}
 [C^i{}_{\Lambda}B_j] = \delta^i_j, \quad
 [C^i{}_{\Lambda}C^j] = 0, \quad
 [B_i{}_{\Lambda}B_j] = 0.
\end{align}
These local SUSY vertex algebras glue classically without obstruction.

\subsection{Quantization, descent, and the anomaly class}
\label{ss:3.4}

We now pass to the quantum theory. 
On a coordinate chart $V \subset X$, 
nothing changes analytically from \S\S\ref{ss:2.9}--\ref{ss:2.10}: 
the theory is still free in the chosen coordinates, 
so we may use the same gauge-fixing $Q^{\GF}$, heat kernel $K_L$, 
propagator $P_{\ep<L}$, and BV Laplacian $\Delta_L$ for $0<\ep<L$. 
Thus, for $L>0$, we obtain a local BD algebra of quantum observables
\begin{align}\label{eq:3.4-quantum-local}
 \qObs{V,L}(U) = \Bigl( \Sym\bigl(\clE_{V,c}(U)^\vee[1]\bigr)\dbr{\hbar}, \; 
                        d^{\tcl}+\hbar\Delta_L \Bigr),
\end{align}
whose factorization product is defined 
by contraction with the regularized propagator. 
Then, we can apply \cref{thm:1:533}, and the cohomology
\begin{align}\label{eq:3:clVqV}
 \clV^q(V) \ceq \bbV\bigl(\qObs{V}\bigr) = H^{\bl}\bigl(\qObs{V}(\sD_1)\bigr)
\end{align}
is the $N_K=1$ SUSY vertex algebra generated by 
the superfields $C^i$ and $B_i$ with OPE
\begin{align}\label{eq:3:OPE}
 C^i(Z_1)\,B_j(Z_2) \sim \hbar \frac{\delta^i_j}{z_{12}}, \quad
 C^i(Z_1)\,C^j(Z_2) \sim 0, \quad
 B_i(Z_1)\,B_j(Z_2) \sim 0.
\end{align}
Thus $\clV^{\q}(V)$ is the $N_K=1$ SUSY vertex algebra of 
the $bc$-$\beta\gamma$ system. 

The essential issue is whether the chart-wise quantum theories glue on overlaps. 
At the classical level, the transition maps were strict; 
quantum mechanically, normal ordering introduces a possible central defect 
on double and triple overlaps. 
More precisely, one obtains on an overlap $V_{\alpha\beta}=V_\alpha \cap V_\beta$ 
a local isomorphism
\begin{align}\label{eq:3:varphi}
 \varphi^{\q}_{\beta\alpha}\colon 
 \rst{\qObs{V_\alpha}}{    V_{\alpha\beta}} \lsto 
      \qObs{V_\beta}\big|_{V_{\alpha\beta}},
\end{align}
whose effect on the generators is still given by the superfield formula
\begin{align}\label{eq:3:qch}
 \Phi^i_\beta  = f^i_{\beta\alpha}(\Phi_\alpha), \quad
 \Pi_{\beta,i} = 
 \frac{\pd x^j_\alpha}{\pd x^i_\beta}(\Phi_\alpha)\,\Pi_{\alpha,j},
\end{align}
but now interpreted inside the completed quantum factorization algebra 
(or equivalently, inside the vertex algebra of local observables 
 with normal ordering understood). 
On triple overlaps $V_{\alpha\beta\gamma}$ these isomorphisms need not satisfy
the cocycle condition strictly; the defect is measured by 
a \v{C}ech $2$-cocycle with values in closed holomorphic $2$-forms.

\begin{thm}\label{thm:3:obs}
Let $X$ be a complex manifold. 
The chart-wise quantum $N=1$ SUSY factorization algebras 
$\{\qObs{V_\alpha}\}_\alpha$ glue canonically on overlaps 
via the superfield coordinate-change formula \eqref{eq:3:qch}
to a sheaf $\sqObs{X}$ of $N=1$ SUSY factorization algebras on $X$.
Applying the local extraction functor of \cref{thm:1:533} produces a sheaf
\[
 \clV^{\q}_X \ceq H^{\bl}\bigl(\sqObs{X}(\sD_1)\bigr)
\]
of $N_K=1$ SUSY vertex algebras on $X$, canonically identified with 
the chiral de Rham complex $\CdR_X$ of $X$ \cite{MSV,GMS3}.
\end{thm}

\begin{proof}
First, we verify the cocycle condition on triple overlaps.
Consider the transition maps $\varphi^{\q}_{\beta\alpha}$ \eqref{eq:3:varphi}
on triple overlaps 
$V_{\alpha\beta\gamma} = V_{\alpha}\cap V_{\beta}\cap V_{\gamma}$.
Let $x_\beta = f_{\beta\alpha}(x_\alpha)$ and 
$x_\gamma = f_{\gamma\beta}(x_\beta)$, then
$\Phi_\gamma=f_{\gamma\beta}(f_{\beta\alpha}(\Phi_\alpha)) 
=f_{\gamma\alpha}(\Phi_\alpha)$,
while for the conjugate field one has
$\Pi_\gamma
= J_{f_{ \gamma\beta}}(\Phi_\beta )^{-1} 
  J_{f_{ \beta\alpha}}(\Phi_\alpha)^{-1} \Pi_\alpha
= J_{f_{\gamma\alpha}}(\Phi_\alpha)^{-1} \Pi_\alpha$,
where $J_f$ denotes the Jacobian matrix of $f$.  
Hence 
$\varphi^{\q}_{\gamma\beta} \circ \varphi^{\q}_{\beta\alpha} 
=\varphi^q_{\gamma\alpha}$ strictly on $V_{\alpha\beta\gamma}$. 
Therefore the local quantum factorization algebras glue to 
a global sheaf $\sqObs{X}$ on $X$.

Let us turn to the identification with the chiral de Rham complex.
By \eqref{eq:3:clVqV}, on each coordinate chart $V \subset X$, 
the SUSY vertex algebra $\clV^{\q}_X(V)$ is the $bc$--$\beta\gamma$ system
generated by superfields $C^i$ and $B_i$ with OPE \eqref{eq:3:OPE}. 
This is exactly the local free-field presentation 
underlying the chiral de Rham complex in \cite{MSV,BHS}. 
Moreover, the local superfields transform \eqref{eq:3:qch} descends to 
\begin{align*}
 C^i_\beta = f^i_{\beta\alpha}(C_\alpha), \quad
 B_{\beta,i} = \frac{\pd x^j_\alpha}{\pd x^i_\beta}(C_\alpha)\,B_{\alpha,j},
\end{align*}
which coincides with the superfield gluing law of Ben-Zvi--Heluani--Szczesny 
\cite[(4.2.2)]{BHS} for the local generators of $\CdR_X$. 
Therefore the descent data defining $\clV^{\q}_X$ agree with 
the descent data defining $\CdR_X$, 
and hence there is a canonical isomorphism of sheaves of $N_K=1$ SUSY vertex algebras
\[
 \clV^{\q}_X \cong \CdR_X.
\]
This proves the theorem.
\end{proof}

\begin{rmk}\label{rmk:3:BHS}
The significance of \cref{thm:3:obs} is that the chiral de Rham complex 
in \cite{MSV,GMS3,BHS} is recovered not merely as a sheaf obtained 
by gluing local vertex algebras, but as the vertex-algebra shadow of a 
\emph{sheaf of $N=1$ SUSY factorization algebras} arising from BV quantization. 
In this sense, the present construction lifts the $N=1$ superfield description 
of \cite{BHS} to the factorization-algebra setting.

If one forgets the fermionic $bc$-part and 
considers only the bosonic curved $\beta\gamma$-system \cite{GGW,Ne,W}, 
then the anomaly class $\mathrm{ch}_2(T_X)$ appearing in 
chiral differential operators (CDO) \cite{GMS1,GMS2}
(equivalently the first Pontryagin class in \cite{B}) appears, 
and one gets the familiar gerbe of chiral differential operators 
unless a trivialization is chosen.
\end{rmk}

\section{Ricci-flat K\"ahler and hyperk\"ahler targets}
\label{s:4}

In this section, we extend the arguments of \S\S\ref{s:2}--\ref{s:3} 
from the case of a general complex target to the case 
where the target carries additional geometric data.
Our aim is to lift the $N_K=2$ and $N_K=4$ supersymmetry of
Ben-Zvi--Heluani--Szczesny from the level of SUSY vertex algebras 
to the level of SUSY factorization algebras. 
More precisely, we show that, whenever $X$ carries a Ricci-flat K\"ahler 
metric or a hyperk\"ahler metric, the local quantum factorization algebra
constructed in \cref{s:3} admits canonical local currents
whose extracted fields reproduce the superconformal generators in \cite{BHS}.

Throughout this section, $X$ is a complex manifold of complex dimension $n$,
and we continue to work locally on the source $\sSig = (\sC,\clD_{\std})$ 
with $\clD_{\std}$ generated by $D_{\std}=\pd_\zeta+\zeta\pd_z$.
We denote by $z_{12} \ceq z_1-z_2-\zeta_1\zeta_2$
the $N_K=1$ super-distance from \cref{eq:1:Z12}. 
For a holomorphic coordinate chart $V \subset X$, we denote by
\begin{align}\label{eq:4:bbVqV}
 \bbV(\qObs{V}) \ceq H^{\bl}\!\bigl(\qObs{V}(\sD_1)\bigr)
\end{align}
the extraction of the local quantum factorization algebra.

\subsection{Local free fields in BHS notation}\label{ss:4.1}

Let $V \subset X$ be a holomorphic coordinate chart 
with holomorphic coordinates $(x^1,\dotsc,x^n)$. 
By \cref{thm:3:obs}, on $V$ the quantum factorization algebra
is locally identified with the free theory of \cref{s:2}. 
In the notation of \cref{cor:2:bcbg}, 
the basic superfields were denoted by $C^i(Z)$ and $B_i(Z)$,
which correspond to the superfields in \cite{BHS} as
$B^i_{\BHS}(Z)=C^i(Z)$ and $\Psi^{\BHS}_i(Z)=B_i(Z)$.
In this section, we also need the anti-holomorphic side, 
which are denoted by 
\[
 \ol{C}^{\ol{i}}(Z),\quad \ol{B}_{\ol{i}}(Z).
\]
Then the only nontrivial local OPEs are
\begin{equation}\label{eq:4:freeOPE}
 C^i(Z_1)B_j(Z_2) \sim \frac{\hbar \delta^i_j}{z_{12}}, \quad
 \ol{C}^{\ol{i}}(Z_1)\ol{B}_{\ol{j}}(Z_2) \sim
 \frac{\hbar \delta^{\ol{i}}_{\ol{j}}}{z_{12}},
\end{equation}
with all other pairings regular.

In terms of the factorization algebras $\qObs{V}$, 
the fields $C^i,B_i,\ol{C}^{\ol{i}},\ol{B}_{\ol{i}}$ are 
the extracted superfields corresponding to linear observables. 
Their classical limits belong to the local Poisson factorization algebra
$\cObs{V}$ from \cref{s:3}.

\subsection{The $N=1$ super-current and the BHS complex-structure current}
\label{ss:4.2}

Let us introduce the basic geometric input 
in the construction of superconformal currents in \cite{BHS}.

Let $J_0$ denote the complex structure of the complex manifold $X$.
Suppose that $X$ is endowed with a Hermitian metric $g$, 
i.e.\ a Riemannian metric on the underlying smooth manifold such that
$g(J_0u,J_0v)=g(u,v)$ for ($u,v \in T_X$).
In local holomorphic coordinates $(x^1,\dots,x^n)$, the metric is written as
\begin{align*}
 g = g_{i\ol{j}}\,dx^i\otimes d\ol{x}^{\ol{j}} +
     g_{\ol{j}i}\,d\ol{x}^{\ol{j}}\otimes dx^i, \quad
 g_{\ol{j}i} = \ol{g_{i\ol{j}}}.
\end{align*}
Equivalently, the associated real $(1,1)$-form is
\[
 \omega \ceq \sqrt{-1}\,g_{i\ol{j}}\,dx^i\wedge d\ol{x}^{\ol{j}}.
\]
When $\omega$ is closed, $(X,g,J_0)$ is K\"ahler. 
In that case, the Levi--Civita connection $\nabla$ of 
the underlying Riemannian metric preserves the complex structure, 
\[
 \nabla J_0 = 0,
\]
and hence preserves the holomorphic and anti-holomorphic splitting 
of the complexified tangent bundle. 

Let $\nabla$ be the Levi--Civita connection of $g$. 
By definition, $\nabla$ is the unique torsion-free connection satisfying $\nabla g=0$. 
In local holomorphic coordinates on a K\"ahler chart, 
the only nonzero Christoffel symbols are
\begin{align*}
 \Gamma^{i}_{jk} = g^{i\ol{\ell}} \pd_j g_{k\ol{\ell}}, \quad
 \Gamma^{\ol{i}}_{\ol{j}\,\ol{k}} = g^{\ell\ol{i}} \pd_{\ol{j}} g_{\ell\ol{k}},
\end{align*}
while the mixed symbols vanish. 
In particular, at every point 
one may choose holomorphic normal coordinates so that
\begin{align*}
 g_{i\ol{j}}=\delta_{i\ol{j}}, \quad
 \Gamma^i_{jk}=0, \quad
 \Gamma^{\ol{i}}_{\ol{j}\,\ol{k}}=0
\end{align*}
at that point. The OPE computations below are performed in such coordinates, and
the resulting formulas glue because they are expressed invariantly in terms of
$g$ and $\nabla$. 

Now, let $V \subset X$ be a holomorphic coordinate chart. 
As in \cref{s:2} and \cref{s:3}, 
let $\Phi=(\Phi^a)_{a=1}^n$ and $\Pi=(\Pi_a)_{a=1}^n$ denote 
the local superfields of the quantum factorization algebra $\qObs{V}$, 
where $\Phi^a$ are coordinate superfields and 
$\Pi_a$ are their conjugate superfields. 
Thus the distinguished currents must be defined already 
inside $\qObs{V}$ in terms of $\Phi$ and $\Pi$, 
and only after applying the local extraction functor $\bbV$ 
do they become superfields in the associated SUSY vertex algebra
(see \cref{thm:2:qObsVA,thm:3:obs}). 

Let $D=\pd_\zeta+\zeta\pd_z$ be the odd translation operator on the source. 
Then the local $N=1$ super-current is
\begin{align}\label{eq:4.2}
 \clH_V(z,\ol{z},\zeta) \ceq \sum_a \nod{\Pi_a\,S\Phi^a}(z,\ol{z},\zeta).
\end{align}
In local holomorphic and anti-holomorphic coordinates
$(x^i,\ol{x}^{\ol{i}})$, this may be written as
\[
 \clH_V(z,\ol{z},\zeta) = \sum_{i=1}^n \nod{\Pi_i\,(D\Phi^i)}(z,\ol{z},\zeta)
     + \sum_{\ol{i}=1}^n \nod{\Pi_{\ol{i}}\,(D\Phi^{\ol{i}})}(z,\ol{z},\zeta).
\]

Now, let $I\in \Gamma(X,\End_{\bbR}(T_X))$ be 
a parallel endomorphism with $\nabla I=0$. 
Following \cite[Lemma 7.2]{BHS}, 
let us consider the following element in $\qObs{V}$:
\begin{align}\label{eq:4.3}
 \clJ_{I,V}(z,\ol{z},\zeta) \ceq 
  \nod{I^a{}_b(\Phi)\,(D\Phi^b)\,\Pi_a}(z,\ol{z},\zeta),
\end{align}
where we used the notational convention 
explained in \cite[(7.1.1)--(7.1.3)]{BHS}.
Since $\nabla I=0$, no additional connection term appears in these coordinates.
If $I=J_0$ is the complex structure of a K\"ahler manifold, 
then in holomorphic coordinates we have 
\[
 J_0(\pd_{x^i})             =  \sqrt{-1}\,\pd_{x^i},\qquad
 J_0(\pd_{\ol{x}^{\ol{i}}}) = -\sqrt{-1}\,\pd_{\ol{x}^{\ol{i}}},
\]
hence
\begin{align}\label{eq:4.4}
 \clJ_V(z,\ol{z},\zeta) = 
  \sqrt{-1} \sum_{i=1}^n \nod{\Pi_i\,(D\Phi^i)}(z,\ol{z},\zeta) -
  \sqrt{-1} \sum_{\ol{i}=1}^n 
            \nod{\Pi_{\ol{i}}\,(D\Phi^{\ol{i}})}(z,\ol{z},\zeta).
\end{align}

\begin{lem}
Let $(X,g,J_0)$ be K\"ahler, 
and let $\nabla$ be the Levi--Civita connection of $g$. 
Then, on each holomorphic chart $V \subset X$, the local observables 
$\clH_V$ and $\clJ_V$ defined by \cref{eq:4.2,eq:4.4} belong to $\qObs{V}$. 
\end{lem}

\begin{proof}
The expressions \cref{eq:4.2,eq:4.4} are Wick-ordered local functionals 
of the basic superfields $\Phi,\Pi$ in $\qObs{V}$, 
hence define local observables.
\end{proof}

Next, we study the OPEs of 
$\clH_V(z,\ol{z},\zeta)$ and $\clJ_V(z,\ol{z},\zeta)$. 
Let
\[
 \rho_g = \sqrt{-1}\,\Ric_{i\ol{j}}\,dx^i\wedge d\ol{x}^{\ol{j}}, \quad
 \Ric_{i\ol{j}} \ceq g^{k\ol{\ell}}R_{i\ol{j}k\ol{\ell}}
\]
be the Ricci form of the metric $g$. 
Define the local observable
\begin{align}\label{eq:4:clA}
 \clA_V(Z) \ceq J_{\rho_g,V}(z,\ol{z},\zeta) 
\end{align}
Using \eqref{eq:4.4},  we have 
\[
 \clA_V(z,\ol{z},\zeta) = 
 \sqrt{-1}\,\Ric_{i\ol{j}}(\Phi(z,\ol{z},\zeta)) 
  \nod{(D\Phi^i)\,(D\Phi^{\ol{j}})}(z,\ol{z},\zeta) +
 \sqrt{-1}\,\Ric^{i\ol{j}}(\Phi(z,\ol{z},\zeta)) 
  \nod{\Pi_i\,\Pi_{\ol{j}}}(z,\ol{z},\zeta),
\]
where
\[
 \Ric^{i\ol{j}} \ceq g^{i\ol{k}}g^{\ell\ol{j}}\Ric_{\ell\ol{k}}.
\]
Now, we denote by $H_V,J_V,A_V \in \bbV(\qObs{V})$  (see \eqref{eq:4:bbVqV}) 
the local observables produced by $\clH_V,\clJ_V,\clA_V$ 
under the extraction functor of \cref{thm:1:533}.

\begin{prp}\label[prp]{prp:4.2}
The singular parts of the OPEs of $H_V,J_V$ satisfy 
the $N=2$ superconformal relations (\cite[\S5.9, p.182]{K}, 
\cite[Examples 2.7, 5.10]{HK}, \cite[Examples 2.5, 3.13]{BHS})
up to the defect field $A_V$:
\[
 \text{(actual singular OPEs of $H_V,J_V$)} =
 \text{(standard $N=2$ superconformal OPEs)} + c\,A_V,
\]
for a constant $c$ depending only on the normalization conventions.
\end{prp}

Equivalently, after absorbing $c$ into the definition of $A_V$, 
the defect is exactly the Ricci-form current. 
In particular,
\[
 A_V=0 \ \iff \ \rho_g=0 \ \iff \  \Ric(g)=0,
\]
and in this case the extracted fields of $H_V$ and $J_V$ satisfy 
the $N=2$ superconformal relations exactly.

\begin{proof}
The expressions \cref{eq:4.2,eq:4.4,eq:4:clA} are Wick-ordered 
local functionals of the basic superfields $\Phi,\Pi$ in $\qObs{V}$, 
hence define local observables.
After applying the extraction functor, $\clH_V$ and $\clJ_V$ become 
the corresponding local superfields in \cite[\S7]{BHS}.

To analyze their OPEs, choose holomorphic normal coordinates for $g$ at a point.
At that point, the Christoffel symbols vanish, 
so the leading singular terms are those of 
the $bc$--$\beta\gamma$ system of \cref{s:2}. 
The deviation from the free-field computation is governed by 
the curvature of the Levi--Civita connection. 
Because $(X,g,J_0)$ is K\"ahler and $\nabla J_0=0$, 
all non-Ricci contractions cancel in the relevant
$H$--$H$, $H$--$J$, and $J$--$J$ singular terms, 
and the remaining defect is the current associated to the Ricci form $\rho_g$,
namely $\clA_V$ from \cref{eq:4:clA}. 
Thus the singular OPEs agree with those of 
the $N=2$ superconformal algebra up to some multiple of $A_V$ claimed above.
\end{proof}

\subsection{The Ricci-flat K\"ahler case}\label{ss:4.3}

Assume that $(X,g,J_0)$ is Ricci-flat K\"ahler. 
Then the Ricci form vanishes: $\rho_g=0$.
Hence, by \cref{prp:4.2}, the defect field $\clA_V = \clJ_{\rho_g,V}$ 
in \eqref{eq:4:clA} vanishes on every holomorphic chart $V \subset X$. 
Therefore the local extracted fields of $H_V$ and $J_V$ 
satisfy the $N=2$ superconformal relations exactly.

We now globalize these local currents. 
By \cref{thm:3:obs}, the chart-wise quantum factorization algebras 
glue to a sheaf $\qObs{X}$ on $X$
Since the local expressions
\[
 \clH_V(z,\ol{z},\zeta) = \nod{ \Pi_a\,(D\Phi^a)}(z,\ol{z},\zeta), \quad
 \clJ_V(z,\ol{z},\zeta) = 
  \nod{ J_0{}^a{}_b(\Phi)\,(D\Phi^b)\,\Pi_a}(z,\ol{z},\zeta)
\]
are defined invariantly from the metric $g$ and 
the parallel complex structure $J_0$, 
they are compatible with the descent maps of \cref{s:3}, 
and hence glue to global local observables of $\qObs{X}$.

\begin{thm}\label{thm:4:N=2}
Let $X$ be a complex manifold endowed with a Ricci-flat K\"ahler metric $g$.
Then:
\begin{enumerate}
\item 
The chart-wise local observables $H_V$ and $J_V$ from
\cref{eq:4.2,eq:4.4} descend to global local observables
\[
 \clH, \, \clJ \in \sqObs{X}(\bbD^{1|1}_1)
\]
for the SUSY disk $\bbD^{1|1}_1$.

\item 
The derivations of $\sqObs{X}$ induced by $\clH$ and $\clJ$ 
satisfy the $N=2$ superconformal relations up to homotopy.

\item 
Under the extraction functor of \cref{thm:1:533}, the cohomology
\[
 \clV^{\q} \ceq \bbV(\sqObs{X}) = 
 H^\bullet\!\bigl(\sqObs{X}(\bbD^{1|1}_1)\bigr)
\]
acquires a natural structure of a sheaf of $N_K=2$ SUSY vertex algebras.

\item 
This $N_K=2$ SUSY vertex algebra agrees locally, and therefore globally,
with the $N=2$ structure on the chiral de Rham complex 
in \cite[Theorem 7.4 (ii)]{BHS}.
\end{enumerate}
\end{thm}

\begin{proof}
The gluing statement in (1) follows from the same descent argument 
as in \cref{thm:3:obs}: 
the currents are expressed invariantly in terms of 
the basic superfields $\Phi,\Pi$ and the parallel tensors $g$ and $J_0$, 
hence are preserved by the coordinate-change isomorphisms. 
By \cref{prp:4.2}, the only local defect in the $N=2$ OPE relations 
is the Ricci-form current $\clA_V = J_{\rho_g,V}$.

Since $g$ is Ricci-flat, we have $\rho_g=0$, 
hence $\clA_V=0$ on every chart. 
Therefore the extracted local fields satisfy 
the $N=2$ superconformal relations exactly. 
This proves (2) and (3), using \cref{thm:1:533} to pass 
from local factorization currents to the associated SUSY vertex algebra. 
Finally, the local formulas coincide with the BHS superfields 
attached to the metric and the complex structure, 
so the resulting $N_K=2$ structure is precisely 
the one constructed in \cite{BHS}.
\end{proof}

\subsection{The hyperk\"ahler currents}\label{ss:4.4}

Assume now that $(X,g,I,J,K)$ is hyperk\"ahler. 
Thus $g$ is a Riemannian metric and $I,J,K$ are parallel complex structures
satisfying the quaternionic relations
\[
 I^2=J^2=K^2=-1, \quad IJ=K=-JI.
\]
For each $A \in \{I,J,K\}$, we define a local observable of $\qObs{V}$ 
for local coordinate $V \subset X$ by the same formula as in \cref{eq:4.3}:
\begin{align}\label{eq:4.6}
 \clJ_{A,V}(z,\ol{z},\zeta) \ceq 
 \nod{A^a{}_b(\Phi)\,(D\Phi^b)\,\Pi_a}(z,\ol{z},\zeta).
\end{align}
Since $A$ is parallel, this expression is coordinate-covariant and hence defines
a local observable on every holomorphic chart $V\subset X$.

Observe that, 
if we choose $A=I$ and take holomorphic coordinates adapted to $I$, 
then \cref{eq:4.6} reduces to 
the K\"ahler current \cref{eq:4.4} of \cref{ss:4.3}. 
The same applies to $J$ and $K$ after changing the holomorphic structure. 
Thus \cref{eq:4.6} is the natural factorization-algebra lift of the currents 
given in \cite[\S7]{BHS} attached to the three hyperk\"ahler complex structures.
Hence, we have:

\begin{prp}\label[prp]{prp:4.3}
Let $(X,g,I,J,K)$ be hyperk\"ahler. 
On each holomorphic chart $V\subset X$, the local observables
\[
 \clH_V,\  \clJ_{I,V},\ \clJ_{J,V},\ \clJ_{K,V}
\]
belong to $\qObs{V}$. 
Under the extraction functor of \cref{thm:1:533}, 
their extracted fields satisfy the local small $N=4$ superconformal relations
(\cite[\S5.10, p.187]{K}, \cite[Example 5.11]{HK}, \cite[Example 2.6]{BHS}).
\end{prp}


Now, we pass from local hyperk\"ahler currents to 
the global factorization algebra.

\begin{thm}\label{thm:4:N=4}
Let $X$ be a hyperk\"ahler manifold. Then:
\begin{enumerate}
\item \label{i:4:HK1}
The local currents $\clH_V,\clJ_{I,V},\clJ_{J,V},\clJ_{K,V}$
descend to global local observables
\[
 H,\ J_I,\ J_J,\ J_K
\]
of the factorization algebra $\sqObs{X}$.

\item \label{i:4:HK2}
The corresponding derivations of $\sqObs{X}$ satisfy the small
$N=4$ superconformal relations up to homotopy.

\item \label{i:4:HK3}
Under the extraction functor of \cref{thm:1:533}, the cohomology
\[
 \bbV(\sqObs{X}) \ceq H^{\bl}\!\bigl(\sqObs{X}(\sD_1)\bigr)
\]
becomes a sheaf of $N_K=4$ SUSY vertex algebras.

\item 
This $N_K=4$ SUSY vertex algebra is the factorization-algebra lift of
the hyperk\"ahler structure on the chiral de Rham complex 
in \cite[Theorem 7.4 (iii)]{BHS}.
\end{enumerate}
\end{thm}

\begin{proof}
Since $I,J,K$ are parallel complex structures, 
the local formula \cref{eq:4.6} is preserved by 
the coordinate-change isomorphisms of \cref{s:3}.
Hence the local currents descend to global local observables of $\sqObs{X}$,
which proves \ref{i:4:HK1}.
By \cref{prp:4.3}, on each chart, 
their extracted fields satisfy the local small $N_K=4$ relations; 
these local identities are preserved under descent and 
therefore define the corresponding global structure 
on the factorization algebra, proving \ref{i:4:HK2}. 
Applying the extraction functor of \cref{thm:1:533} 
then yields the $N_K=4$ SUSY vertex algebra in \ref{i:4:HK3}. 
The comparison with \cite[\S7]{BHS} follows from the fact 
that the local superfields are exactly the currents 
in \cite[Theorem 7.4 (iii)]{BHS}
attached to the three hyperk\"ahler complex structures.
\end{proof}

\begin{Ack}
The author was supported by JSPS KAKENHI Grant Number 26K06762 
and by Asahipen Hikari Foundation.
\end{Ack}


\end{document}